\documentclass[aos]{imsart}

\RequirePackage{amsthm,amsmath,amsfonts,amssymb,mathrsfs}
\usepackage{textcomp}
\usepackage{mathrsfs}
\usepackage{amsmath}
\usepackage{amsthm}
\usepackage{amssymb}
\usepackage{cancel}

\makeatletter

\RequirePackage[numbers]{natbib}
\RequirePackage[colorlinks,citecolor=blue,urlcolor=blue]{hyperref}
\RequirePackage{graphicx,cancel}

\startlocaldefs
\theoremstyle{plain}

\newtheorem{cor}{Corollary}[section]

\newtheorem{theorem}{Theorem}[section]
\newtheorem{lemma}[theorem]{Lemma}
\theoremstyle{remark}
\newtheorem{definition}[theorem]{Definition}
\newtheorem{example}{Example}

\endlocaldefs
\begin{document}

\begin{frontmatter} 
\title{Inadmissibility and Transience}

\maketitle
\runtitle{Inadmissibility and Transience} 

\begin{aug} 
\author[A]{\fnms{K\={o}saku} \snm{Takanashi}\ead[label=e1]{kosaku.takanashi@riken.jp }}
\and
\author[B]{\fnms{Kenichiro} \snm{McAlinn}\ead[label=e3,mark]{kenichiro.mcalinn@temple.edu}}
\address[A]{RIKEN,
Center for Advanced Intelligence Project,
\printead{e1}}

\address[B]{Department of Statistics, Operations, and Data Science,
Fox School of Business, Temple University,
\printead{e3}}
 \end{aug}
 
\begin{abstract}
We discuss the relation between the statistical question of inadmissibility
and the probabilistic question of transience. \cite{Brown_71} proved
the mathematical link between the admissibility of the mean of a Gaussian
distribution and the recurrence of a Brownian motion, which holds
for $\mathbb{R}^{2}$ but not for $\mathbb{R}^{3}$ in Euclidean space.
We extend this result to symmetric, non-Gaussian distributions, without
assuming the existence of moments. As an application, we prove that
the relation between the inadmissibility of the predictive density
of a Cauchy distribution under a uniform prior and the transience
of the Cauchy process differs from dimensions $\mathbb{R}^{1}$ to
$\mathbb{R}^{2}$. We also show that there exists an extreme model
that is inadmissible in $\mathbb{R}^{1}$. 
\end{abstract}
\begin{keyword}[class=MSC] \kwd[Primary ]{62C15} 
\kwd[; secondary ]{60J46} {62C10} {62C20} \end{keyword}

\begin{keyword}  \kwd{Inadmissibility} \kwd{Dirichlet form} \kwd{Cauchy distribution} \kwd{Stable distribution} \kwd{Bayesian predictive distribution} \end{keyword}

\end{frontmatter} 

\section{Introduction}

Consider the mathematical connection between the admissibility of
an estimator and the recurrence of a stochastic process. On the former,
we have the classical statistical decision problem of the admissibility
of an estimator. Since the seminar work of \cite{Stein_55}, the admissibility,
or inadmissibility, of an estimator with regard to its dimension has
received considerable interest. In the case of the Gaussian distribution,
\cite{Stein_55} proved that the best equivariant estimator is admissible
if the dimension of the multivariate Gaussian, $d$, satisfies $d\leq2$,
and inadmissible if $d\geq3$. On the latter, we have the probabilistic
problem of the recurrence of a stochastic process. In an analog
to the estimation of the Gaussian mean, a Brownian motion is known
to be recurrent if $d=1,2$ and transient if $d\geq3$. Connecting
the two problems, and establishing a mathematical link between admissibility
and recurrence, has received much focus. 

For the Gaussian case, \cite{Brown_71} proved that the estimator
for the mean is admissible if and only if the corresponding diffusion
process is recurrent. Specifically, the transience of the diffusion
process implies inadmissibility, without any regularity conditions,
while the other direction is rigorously established when the estimator
has bounded risk. For non-Gaussian cases, \cite{Johnstone_84} provided
the Poisson counterpart to \cite{Brown_71} by proving that the admissibility
of an estimator for the mean corresponds to the recurrence of a unique
reversible birth and death process. The key idea in both papers is
to reformulate the question of admissibility to a calculus of variational
minimization problem. Specifically, it is known that the diffusion
is recurrent if and only if the exterior Dirichlet problem is insoluble.
Exploiting the connection to the statistical problem, these papers
show that the statistical estimator is admissible, equivalently, if
and only if the exterior Dirichlet problem is insoluble. More generally,
\cite{Eaton_92} proved that, for general distributions, the sufficient
condition of admissibility is for the associated stochastic process
to be recurrent.

While \cite{Brown_71} and \cite{Johnstone_84} proved the equivalence
connection between the admissibility/inadmissibility of an estimator
and the recurrence/transience for the Gaussian and Poisson processes,
respectively, whether this equivalence holds for general distributions
has remained an open question. One specific problem of interest is
the Cauchy distribution. Just as recurrence/transience is determined
by the dimension at $\mathbb{R}^{2}$ and $\mathbb{R}^{3}$ for the
Brownian motion, it is also known that the recurrence/transience of
a Cauchy process differs at $\mathbb{R}^{1}$ and $\mathbb{R}^{2}$.
Similar to the case in \cite{Brown_71}, one can expect the prediction
of the associated distribution-- the Cauchy-- to also be admissible
at $\mathbb{R}^{1}$ and inadmissible at $\mathbb{R}^{2}$. As a more
extreme case, consider the estimation of a stable process with index
$\alpha$ in $\mathbb{R}^{1}$. It is known that this is transient
when $\alpha<1$. Extending the concept of \cite{Brown_71} to a stable
process, one can expect that the predictive distribution using the MLE as the location estimate is
inadmissible even at $\mathbb{R}^{1}$.

Our contribution is to generalize the result in \cite{Brown_71} to
the relation between the symmetric infinite divisible distribution and the Lévy
process. We prove that the statistical problem of the inadmissibility
of estimating the predictive density of a Cauchy at $\mathbb{R}^{2}$
is equivalent to the probabilistic problem of the transience of a
Cauchy process at $\mathbb{R}^{2}$. A critical departure from previous
research is that we cannot assume the existence of moments, since
we cannot define the squared error loss of the mean for the Cauchy.
For this paper, we evaluate the Kullback-Leibler loss. This is suitable,
since the evaluation is of the Bayes predictive distribution and MLE
plug-in distribution, which can be interpreted as minimizing the Kullback-Leibler
loss.

\subsection{Notation}

We use upper case ($X$) for random variables and lower case ($x$)
for its realizations. Random variables with subscripts, $s,t$, denote
stochastic processes (e.g. $X_{t}$). $X\overset{\textrm{d}}{=}Y$
denotes that $X,Y$ are identically distributed. $\left\langle \cdot,\cdot\right\rangle ,\left\Vert \cdot\right\Vert $
are each the inner product and norm under $\mathbb{R}^{d}$, respectively.
$L^{2}\left(\mathbb{R}^{d};m\right)$ denotes the square integrable
function space regarding measure $m$ on $\mathbb{R}^{d}$. $L^{1}\left(\mathbb{R}^{d};m\right),L_{+}^{1}\left(\mathbb{R}^{d};m\right)$
are each the integrable function space and non-negative integrable
function space regarding measure $m$ on $\mathbb{R}^{d}$. When written
as $\left\langle \cdot,\cdot\right\rangle _{m}$, it is the inner
product on $L^{2}\left(\mathbb{R}^{d};m\right)$. $\mathsf{m}$ is
the Lebesgue measure on $\mathbb{R}^{d}$.

\section{Main Results}

\subsection{Bayes risk difference and Blyth's method}

Consider a symmetric $d$-dimensional location-scale distribution,
$p_{c}\left(x\left|\theta\right.\right)$, with location, $\theta$,
and scale, $cI$ ($I$ is a $d\times d$ identity matrix). Here, symmetric
means that the stochastic variable, $X\sim p_{c}\left(x\left|\theta\right.\right)$,
is equivalent around the location, $\theta$; i.e. $-\left(X-\theta\right)\overset{\textrm{d}}{=}X-\theta$.
Assumptions about the moments, including their existence, are not
made. Based on observing a datum, $X=x$, we consider the problem
of obtaining a predictive density, $\hat{p}\left(y\left|x\right.\right)$,
for $Y$ that is close to $p_{c}\left(y\left|\theta\right.\right)$.
We measure this closeness by the Kullback-Leibler (KL) loss, 
\[
L_{\mathsf{KL}}\left(\theta,\hat{p}\left(\cdot\left|x\right.\right)\right)=\mathsf{KL}\left(p_{c}\left|\hat{p}\right.\right)=\int_{\mathbb{R}^{d}}\log\frac{p_{c}\left(y\left|\theta\right.\right)}{\hat{p}\left(y\left|x\right.\right)}p_{c}\left(y\left|\theta\right.\right)dy,
\]
and evaluate $\hat{p}$ by its expected loss or risk function, 
\[
R_{\mathsf{KL}}\left(\theta,\hat{p}\right)=\int_{\mathbb{R}^{d}}L_{\mathsf{KL}}\left(\theta,\hat{p}\left(\cdot\left|x\right.\right)\right)p_{c}\left(x\left|\theta\right.\right)dx.
\]
For the comparison of two procedures, we say that $\hat{p}^{1}$ dominates
$\hat{p}^{2}$ if $R_{\mathsf{KL}}\left(\theta,\hat{p}^{1}\right)\leqq R_{\mathsf{KL}}\left(\theta,\hat{p}^{2}\right)$
for all $\theta$ and with strict inequality for some $\theta$. The
Bayes risk 
\[
B_{\mathsf{KL}}\left(\pi,\hat{p}\right)=\int_{\mathbb{R}^{d}}R_{\mathsf{KL}}\left(\theta,\hat{p}\right)\pi\left(\theta\right)d\theta
\]
is minimized by the predictive distribution 
\[
\hat{p}^{\pi}\left(y\left|x\right.\right)=\int_{\mathbb{R}^{d}}p_{c}\left(y\left|\theta\right.\right)\pi\left(\theta\left|x\right.\right)d\theta.
\]
Unless $\pi$ is a trivial point prior, we have 
\[
\hat{p}^{\pi}\left(y\left|x\right.\right)\notin\left\{ p_{c}\left(y\left|\theta\right.\right):\theta\in\mathbb{R}^{p}\right\} ,
\]
where $\hat{p}^{\pi}$ will not correspond to a \textquotedblright plug-in\textquotedblleft{}
estimate, $\hat{p}^{\pi}$. In other words, the predictive distribution
dominates any plug-in estimate \citep{Aitchison_75}. The best equivariant
predictive density, with respect to the location group, is the Bayes
predictive density under the uniform prior, $\pi_{U}\left(\theta\right)\equiv1$,
which has constant risk \citep[see]{Murray_77,Ng_80}. More precisely,
one might refer to $\hat{p}^{\pi_{U}}$ as a generalized Bayes solution
because $\pi_{U}$ is improper. \cite{Aitchison_75} showed that $\hat{p}^{\pi_{U}}\left(y\left|x\right.\right)$
dominates the plug-in predictive density $p_{c}\left(y\left|\hat{\theta}_{MLE}\right.\right)$,
which simply substitutes the MLE $\hat{\theta}_{MLE}=x$ for $\theta$.
This is why the evaluation must be done on the predictive density
under KL loss.

Again, the KL risk and KL Bayes risk are 
\begin{alignat*}{1}
R_{\mathsf{KL}}\left(\theta,\hat{p}^{\pi}\right) & =\int_{\mathbb{R}^{d}}\left\{ \int_{\mathbb{R}^{d}}\log\frac{p_{c}\left(y\left|\theta\right.\right)}{\hat{p}^{\pi}\left(y\left|x\right.\right)}p_{c}\left(y\left|\theta\right.\right)dy\right\} p_{c}\left(x\left|\theta\right.\right)dx\\
B_{\mathsf{KL}}\left(\pi,\hat{p}^{\pi}\right) & =\int_{\mathbb{R}^{d}}\int_{\mathbb{R}^{d}}\left\{ \int_{\mathbb{R}^{d}}\log\frac{p_{c}\left(y\left|\theta\right.\right)}{\hat{p}^{\pi}\left(y\left|x\right.\right)}p_{c}\left(y\left|\theta\right.\right)dy\right\} p_{c}\left(x\left|\theta\right.\right)dx\cdot\pi\left(\theta\right)d\theta.
\end{alignat*}
The Bayes risk difference is 
\begin{alignat*}{1}
B_{\mathsf{KL}}\left(\pi,\hat{p}^{\pi_{U}}\right)-B_{\mathsf{KL}}\left(\pi,\hat{p}^{\pi}\right) & =\int_{\mathbb{R}^{d}}\int_{\mathbb{R}^{d}}\left\{ \int_{\mathbb{R}^{d}}\log\frac{\hat{p}^{\pi}\left(y\left|x\right.\right)}{\hat{p}^{\pi_{U}}\left(y\left|x\right.\right)}p_{c}\left(y\left|\theta\right.\right)dy\right\} p_{c}\left(x\left|\theta\right.\right)dx\cdot\pi\left(\theta\right)d\theta\\
 & =\int_{\mathbb{R}^{d}}\int_{\mathbb{R}^{d}}\log\frac{\hat{p}^{\pi}\left(y\left|x\right.\right)}{\hat{p}^{\pi_{U}}\left(y\left|x\right.\right)}\hat{p}^{\pi}\left(y\left|x\right.\right)M^{\pi}\left(x;c\right)dxdy.
\end{alignat*}
Here, $M^{\pi}\left(x;c\right)$ is the marginal distribution under
$\pi$: 
\[
M^{\pi}\left(x;c\right)=\int_{\mathbb{R}^{d}}p_{c}\left(x\left|\theta\right.\right)\pi\left(\theta\right)d\theta.
\]
Further, since we assume that the scale parameter, $cI$, is equivalent
for $X,Y$, we have 
\begin{alignat*}{1}
\int_{\mathbb{R}^{d}}\hat{p}^{\pi}\left(y\left|x\right.\right)M^{\pi}\left(x;c\right)dx & =\int_{\mathbb{R}^{d}}\left\{ \int_{\mathbb{R}^{d}}p_{c}\left(y\left|\theta\right.\right)\frac{\pi\left(\theta\right)p_{c}\left(\theta\left|x\right.\right)}{M^{\pi}\left(x;c\right)}d\theta\right\} M^{\pi}\left(x;c\right)dx\\
 & =\int_{\mathbb{R}^{d}}p_{c}\left(y\left|\theta\right.\right)\left\{ \int_{\mathbb{R}^{d}}p_{c}\left(\theta\left|x\right.\right)dx\right\} \pi\left(\theta\right)d\theta,
\end{alignat*}
which implies that $M^{\pi}\left(x;c\right)$ is the invariant distribution
of $\hat{p}^{\pi}\left(y\left|x\right.\right)$: 
\begin{equation}
\int_{\mathbb{R}^{d}}\hat{p}^{\pi}\left(y\left|x\right.\right)M^{\pi}\left(y;c\right)dy=M^{\pi}\left(x;c\right).\label{eq:Marginal}
\end{equation}
From this, we have 
\begin{equation}
B_{\mathsf{KL}}\left(\pi,\hat{p}^{\pi_{U}}\right)-B_{\mathsf{KL}}\left(\pi,\hat{p}^{\pi}\right)=\int_{\mathbb{R}^{d}}\mathsf{KL}\left(\hat{p}^{\pi}\left|\hat{p}^{\pi_{U}}\right.\right)M^{\pi}\left(x;c\right)dx.\label{eq:BayesRiskDiffe}
\end{equation}
To examine the admissibility of $\hat{p}^{\pi_{U}}$ under KL risk,
Blyth's method is an effective strategy, as with \cite{Brown-EdGeorge-Xu_08}
Lemma 1. Therefore, it is sufficient to ask whether there exists a
measure sequence of a proper prior, $\left\{ \pi_{n}\right\} $, such that $B_{\mathsf{KL}}\left(\pi,\hat{p}^{\pi_{U}}\right)-B_{\mathsf{KL}}\left(\pi,\hat{p}^{\pi_n}\right)$ (the Bayes risk difference using $\left\{ \pi_{n}\right\} $ as a proper prior for the proper predictive distribution, $\hat{p}^{\pi_n}$) converges to zero.

\subsection{Equivalence between Bayes risk difference and Dirichlet form \label{sec:2.2}}

Our main contribution is to show that the Bayes risk difference Eq.~(\ref{eq:BayesRiskDiffe})
in the above setup satisfies the following equality:

\begin{theorem} \label{prop:BayesRisk} Let $\mathsf{m}$ be a Lebesgue
measure on $\mathbb{R}^{d}$. Further, let $M^{\pi}$ be the Radon-Nikodym
derivative regarding the Lebesgue measure, $\mathsf{m}$, of the invariant
measure, $M^{\pi}\left(x;c\right)dx$, of $\hat{p}^{\pi}$. The Bayes
risk difference Eq.~(\ref{eq:BayesRiskDiffe}) satisfies 
\begin{equation}
\int_{\mathbb{R}^{d}}\mathsf{KL}\left(\hat{p}^{\pi}\left|\hat{p}^{\pi_{U}}\right.\right)M^{\pi}\left(x;c\right)dx=\mathcal{E}\left(\sqrt{M^{\pi}},\sqrt{M^{\pi}}\right).\label{eq:BRD=00003D00003D00003DDirichlet}
\end{equation}
Here, $\mathcal{E}\left(f,f\right)$ and $f\in\mathcal{F}$ are the
values of the Dirichlet form of a Markov process with $\hat{p}^{\pi_{U}}$
as its transition probability and the domain of its Dirichlet form.
Details of the Dirichlet form are defined below.\end{theorem}

Here, $\mathcal{E}\left(\cdot,\cdot\right)$ is the Dirichlet form
with transition probability, $\hat{p}^{\pi_{U}}$, and is equal to
the quadratic variation of a Markov process, $\left\{ X_{t}\right\} $,
with transition probability, $\hat{p}^{\pi_{U}}$, under the transformation
of variable, $\sqrt{M^{\pi}}$  (details regarding the connection between the predictive distribution and the continuous-time Markov process are given in Section~\ref{subsec:Markov-Bayes}). Intuitively, this is equal to the
instantaneous variance: 
\[
\lim_{\varDelta t\rightarrow0}\frac{1}{\varDelta t}\mathbb{E}_{\hat{p}^{\pi_{U}}}\left[\left(\sqrt{M^{\pi}\left(X_{t+\varDelta t}\right)}-\sqrt{M^{\pi}\left(X_{t}\right)}\right)^{2}\right],
\]
where $\mathbb{E}_{\hat{p}^{\pi_{U}}}$ represents expectation with
respect to the transition probability $\hat{p}^{\pi_{U}}$. Eq.~(\ref{eq:BRD=00003D00003D00003DDirichlet})
is a generalization of \cite{Brown_71} Eq.~(1.3.4) by extending
the quadratic loss to KL loss, and normal distribution to a symmetric
location model.

For the case of KL loss with normal distributions, \cite{Brown-EdGeorge-Xu_08}
Corollary 1. shows the equivalence between the Bayes risk difference
and Dirichlet form. The result states that when $X\sim N\left(\mu,v_{x}I\right)$
and $Y\sim N\left(\mu,v_{y}I\right)$, we have the following equality:
\begin{equation}
B_{\mathsf{KL}}\left(\mu,\hat{p}_{\pi_{U}}\right)-B_{\mathsf{KL}}\left(\mu,\hat{p}_{\pi}\right)=\frac{1}{2}\int_{v_{w}}^{v_{x}}\frac{1}{v^{2}}\left[B_{Q}^{v}\left(\mu,\hat{\mu}_{\textrm{MLE}}\right)-B_{Q}^{v}\left(\mu,\hat{\mu}_{\pi}\right)\right]dv\label{eq:BGX-08_1}
\end{equation}
where, 
\begin{alignat*}{1}
B_{Q}^{v}\left(\mu,\hat{\mu}\right) & =\int_{\mathbb{R}^{d}}\mathbb{E}_{N_{p}\left(\mu,vI\right)}\left[\left\Vert \hat{\mu}-\mu\right\Vert ^{2}\right]\pi\left(\mu\right)d\mu,\\
 & \hat{\mu}_{\pi}=\int_{\mathbb{R}^{d}}\mu\pi\left(\left.\mu\right|x\right)d\mu,\\
 & v_{w}=\frac{v_{x}v_{y}}{v_{x}+v_{y}}.
\end{alignat*}
The $\frac{1}{v^{2}}\left[B_{Q}^{v}\left(\mu,\hat{\mu}_{\textrm{MLE}}\right)-B_{Q}^{v}\left(\mu,\hat{\mu}_{\pi}\right)\right]$
inside the integral on the right-hand side in Eq.~\eqref{eq:BGX-08_1}
is four times the Dirichlet form of the standard Brownian motion \citep[from][Eq.~(1.3.4)]{Brown_71}:
\begin{alignat}{1}
\frac{1}{v^{2}}\left[B_{Q}^{v}\left(\pi,\hat{\mu}_{\textrm{MLE}}\right)-B_{Q}^{v}\left(\pi,\hat{\mu}_{\pi}\right)\right] & =4\int_{\mathbb{R}^{d}}\left\Vert \nabla_{z}\sqrt{m^{\pi}\left(z;v\right)}\right\Vert ^{2}dz.\label{eq:Brown-71}
\end{alignat}
Here, $m^{\pi}\left(z;v\right)$ is the marginal likelihood of $p\left(\left.z\right|\mu\right)=N\left(\mu,vI\right)$
under the prior, $\pi\left(\mu\right)$. Finally, we have 
\begin{equation}
B_{\mathsf{KL}}\left(\mu,\hat{p}_{\pi_{U}}\right)-B_{\mathsf{KL}}\left(\mu,\hat{p}_{\pi}\right)=2\int_{v_{w}}^{v_{x}}\left[\int_{\mathbb{R}^{d}}\left\Vert \nabla_{z}\sqrt{m^{\pi}\left(z;v\right)}\right\Vert ^{2}dz\right]dv.\label{eq:BGX-08_2}
\end{equation}

Although Eqs.~\eqref{eq:BRD=00003D00003D00003DDirichlet}
and \eqref{eq:BGX-08_2} do not, initially, look equivalent,
when $v_{x}=v_{y}$, we have $v_{w}=\frac{1}{2}v_{x}$, making them equivalent.

\begin{lemma} \label{lem:BGX-Cor1}When $X\sim N\left(\mu,v_{x}I\right)$,
$Y\sim N\left(\mu,v_{y}I\right)$, and $v_{x}=v_{y}$, we have $v_{w}=\frac{1}{2}v_{x}$,
and the following holds: 
\[
\frac{1}{2}\int_{v_{w}}^{v_{x}}\frac{1}{v^{2}}\left[B_{Q}^{v}\left(\pi,\hat{\mu}_{\textrm{MLE}}\right)-B_{Q}^{v}\left(\pi,\hat{\mu}_{\pi}\right)\right]dv=\mathcal{E}_{\textrm{BM}}^{2v_{x}}\left(\sqrt{m_{v_{x}}^{\pi}},\sqrt{m_{v_{x}}^{\pi}}\right),
\]
where $m_{v_{x}}^{\pi}=m^{\pi}\left(x;v_{x}\right)$. The Dirichlet
form, $\mathcal{E}_{\textrm{BM}}^{2v_{x}}\left(\cdot,\cdot\right)$,
corresponds to a Brownian motion where the transition probability
is the predictive distribution, $\hat{p}_{\pi_{U}}$. \end{lemma}

From this, we can see that the results for normal distributions in
\cite{Brown-EdGeorge-Xu_08} also hold within the framework of this
paper. The converse is also true, in that our results hold under the
setting in \cite{Brown-EdGeorge-Xu_08} when $v_{x}=v_{y}$.

Now, for \cite{Brown-EdGeorge-Xu_08} Theorem 1., they evaluate a
stronger KL risk, 
\[
R_{\mathsf{KL}}\left(\mu,\hat{p}_{\pi_{U}}\right)-R_{\mathsf{KL}}\left(\mu,\hat{p}_{\pi}\right)=\frac{1}{2}\int_{v_{w}}^{v_{x}}\frac{1}{v^{2}}\left[R_{Q}^{v}\left(\mu,\hat{\mu}_{\textrm{MLE}}\right)-R_{Q}^{v}\left(\mu,\hat{\mu}_{\pi}\right)\right]dv
\]
compared to Eq.~\eqref{eq:BGX-08_1}. Because of this, they not only
derive the condition for admissibility and inadmissibility, they also
derive the sufficient condition to dominate $\hat{p}_{\pi_{U}}$.
Unfortunately, we were not able to derive the sufficient condition
to construct an estimator or predictive distribution that is superior
when inadmissible. As the main purpose of this paper is to reduce
the Bayes risk difference to the analytic Dirichlet form, we consider the
derivation of a dominating uniform prior predictive density, $\hat{p}_{\pi_{U}}$, to be future research.

From equality Eq.~\eqref{eq:BRD=00003D00003D00003DDirichlet}, we
can connect the statistical decision problem and the Markov process. Whether
the Markov process is recurrent or transient can be discerned by the
Dirichlet form, and is dependent on the existence of a sequence of
functions where the Dirichlet form becomes zero \citep{Fukushima-Oshima-Takeda_10}.
Thus, we have 
\begin{itemize}
\item The necessary and sufficient condition for the recurrence of a Markov
process with transition probability $\hat{p}_{t}^{\pi_{U}}$ is the
existence of a function sequence, $\left\{ f_{n}\right\} $, that
satisfies 
\[
\left\{ f_{n}\right\} \subset\mathcal{F},\quad\lim_{n\rightarrow\infty}f_{n}=1\left(\mathsf{m}\textrm{-a.e.}\right),\quad\lim_{n\rightarrow\infty}\mathcal{E}\left(f_{n},f_{n}\right)=0.
\]
\item The necessary and sufficient condition for the transience of a Markov
process with transition probability $\hat{p}_{t}^{\pi_{U}}$ is that
there exists an $\mathsf{m}$-integrable function $g$ that is bounded
on $\mathbb{R}^{d}$ with $g>0,\ \mathsf{m}\textrm{-a.e.}$ and satisfies
\begin{alignat*}{2}
0<\int_{\mathbb{R}^{d}}\left|f\right|gd\mathsf{m}\leqq & \mathcal{E}\left(\sqrt{f},\sqrt{f}\right), & \forall\sqrt{f}\in\mathcal{F}.
\end{alignat*}
\end{itemize}
With the above theorems, we can correspond Blyth's method and the
discernment of recurrence and transience of a Markov process through
its Dirichlet form. This is why the admissibility or inadmissibility
of $\hat{p}^{\pi_{U}}$ under KL risk has a corresponding relationship
with the recurrence or transience of a Markov process with transition
probability $\hat{p}^{\pi_{U}}$. For the $d$-dimensional Cauchy
process, it is known that the recurrence and transience switches between
$d=1$ and $d\geqq2$ \citep{Sato_99}, which leads to the following
analogue of \cite{Brown_71} for the Cauchy distribution:

\begin{cor} \label{cor:Cauchy} Consider a $d$-dimensional Cauchy
distribution, $\mathcal{C}\left(\theta,cI\right)$, with unknown location,
$\theta$, and known scale, $c$. For the problem of estimating the
predictive distribution under KL risk, upon observing $X=x$, 
\begin{itemize}
\item When $d=1$, the MLE, $\hat{\theta}=x$, plug-in predictive distribution
and the uniform prior Bayes predictive distribution is admissible. 
\item When $d\geqq2$, the MLE, $\hat{\theta}=x$, plug-in predictive distribution
and the uniform prior Bayes predictive distribution is inadmissible. 
\end{itemize}
\end{cor}

Further, for the 1-dimensional stable distribution, $\textrm{Stable}\left(\alpha,\gamma,c\right)$,
if the index, $\alpha$, is less than one, it is inadmissible.

\begin{cor} \label{cor:Stable} Consider a $1$-dimensional stable
distribution, $\emph{Stable}\left(\alpha,\gamma,c\right)$, with unknown
location, $\gamma$, and known scale, $c$. For the problem of estimating
the predictive distribution under KL risk, upon observing $X=x$, 
\begin{itemize}
\item When $\alpha=1$, $\emph{Stable}\left(\alpha,\gamma,c\right)$ is
a Cauchy distribution and thus the MLE, $\hat{\theta}=x$, plug-in
predictive distribution and the uniform prior Bayes predictive distribution
is admissible. 
\item When $0<\alpha<1$, the MLE, $\hat{\theta}=x$, plug-in predictive
distribution and the uniform prior Bayes predictive distribution is
inadmissible. 
\end{itemize}
\end{cor}

While \cite{Eaton_92} also considers the relationship between a Markov
process and the problem of estimating the mean parameter, our approach
is notably different in several ways. First, the loss function in
\cite{Eaton_92} is quadratic, which assumes the second moment of
the estimator. Since this is problematic for the Cauchy, we use the
KL loss and do not assume the existence of moments. Second, the goal
of \cite{Eaton_92} is to derive the sufficient condition for admissibility
by bounding the Bayes risk difference from above using the Dirichlet
form. We, on the other hand, derive the exact relation between the
Bayes risk difference and Dirichlet form to provide the necessary
and sufficient condition to discern admissibility and inadmissibility,
similar to \cite{Brown_71}.

\subsection{Sufficient conditions for admissibility\label{subsec:Sufficient-conditions}}

In this section, we will show that the $\mathbb{R}^{1}$-Cauchy distribution is admissible. Specifically, we show
the following: the group invariant prior is a Lebesgue measure for the
location parameter, but whether an estimator is inadmissible can be
known from Eq.~\eqref{eq:BRD=00003D00003D00003DDirichlet} when the
predictive distribution based on the Lebesgue prior is transient.
Conversely, when the predictive distribution is recurrent, there exists
a way to construct a prior sequence that approximates a uniform prior
to make the estimator admissible. The $\mathbb{R}^{1}$-Cauchy distribution
in Corollary~\ref{cor:Cauchy} and $\mathbb{R}^{1,2}$-normal distribution
are examples of this construction.

\begin{lemma} \label{lem:Diri<prior} Let $\left(\mathcal{E},\mathcal{F}\right)$
be the Dirichlet form that follows a Markov process with transition probability, $\hat{p}_{t}^{\pi_{U}}$.
Then, we have $\sqrt{M^{\pi}}\in\mathcal{F}$, and for $\mathcal{E}\left(\sqrt{M^{\pi}},\sqrt{M^{\pi}}\right)$
we have the following inequality, 
\[
\mathcal{E}\left(\sqrt{M^{\pi}},\sqrt{M^{\pi}}\right)\leqq\mathcal{E}\left(\sqrt{\pi},\sqrt{\pi}\right).
\]

\end{lemma} \begin{proof} See Supplementary Material Appendix~\ref{app:Diri<prior}.
\end{proof}

This lemma provides a guide to constructing a prior distribution that
makes the estimator admissible. From this Lemma, to show admissibility,
we need to show that a sequence of proper priors, $\left\{ \pi_{n}\right\} $,
such that $\mathcal{E}\left(\sqrt{\pi_{n}},\sqrt{\pi_{n}}\right)\rightarrow0$,
can be constructed. However, $\left\{ \pi_{n}\right\} $ is a functional
sequence in the $L^{2}\left(\mathbb{R}^{d},\mathsf{m}\right)$ space.
Therefore, it is a functional sequence of a proper prior, and we construct
a functional sequence such that it is a uniform prior when $n\rightarrow\infty$.
When $\left(\mathcal{E},\mathcal{F}\right)$ is recurrent, the functional
sequence can be constructed in $\mathcal{F}$, but cannot when it
is transient. This is the distinction between admissibility/inadmissibility
for the $\mathbb{R}^{1},\mathbb{R}^{d}\left(d\geqq2\right)$ Cauchy
distribution (\ref{cor:Cauchy}), and $\mathbb{R}^{1,2},\mathbb{R}^{d}\left(d\geqq3\right)$
normal distribution.

We now construct such a sequence. Assume that the Dirichlet form,
$\left(\mathcal{E},\mathcal{F}\right)$, of the Markovian semigroup,
$\left\{ T_{t}\right\} $, that corresponds to the transition probability,
$\hat{p}^{\pi_{U}}$, and $\hat{p}^{\pi_{U}}$ is recurrent. We transform
the Markovian semigroup with transition probability, $\hat{p}^{\pi_{U}}$,
as $\left\{ T_{t}^{\eta}\right\} $. Choose a function, $\eta$, that
is $\eta\in L^{1}\left(\mathbb{R}^{d};\mathsf{m}\right)\cap L^{\infty}\left(\mathbb{R}^{d};\mathsf{m}\right),\ \eta>0\ \mathsf{m}\textrm{-a.e.}$,
and define it as 
\[
T_{t}^{\eta}f\left(x\right)\triangleq\int_{\mathbb{R}^{d}}\exp\left(-t\eta\left(y\right)\right)\hat{p}_{t}^{\pi_{U}}\left(y\left|x\right.\right)f\left(y\right)dy,\ f\in L^{2}\left(\mathbb{R}^{d};\mathsf{m}\right).
\]
The corresponding Dirichlet form is 
\[
\mathcal{E}^{\eta}\left(f,g\right)=\mathcal{E}\left(f,g\right)+\left\langle f,g\right\rangle _{\eta\cdot\mathsf{m}},\ f,g\in\mathcal{F}.
\]
Although, 
\[
\left\langle f,g\right\rangle _{\eta\cdot\mathsf{m}}=\int_{\mathbb{R}^{d}}f\left(x\right)g\left(x\right)\eta\left(x\right)\mathsf{m}\left(dx\right).
\]
Here, $\left\{ T_{t}^{\eta}\right\} $ is a Markov process semigroup
that is generated by a particle following a Markov process corresponding
to $\left\{ T_{t}\right\} $ that has a survival probability $\int_{0}^{t}\exp\left(-s\eta\left(y\right)\right)ds$,
at $t$.

Denote the transition semigroup of this transformed transition probability
as $\left\{ T_{t}^{\eta}\right\} $, resolvent as $G_{\alpha}^{\eta}$,
and consider the function, $G_{\alpha}^{\eta}\eta$. Therefore, 
\begin{alignat*}{1}
G_{\alpha}^{\eta}\eta\left(x\right) & =\int_{0}^{\infty}e^{-\alpha s}T_{s}\eta\left(x\right)ds.\\
 & =\int_{0}^{\infty}e^{-\alpha s}\int_{\mathbb{R}^{d}}\exp\left(-s\eta\left(y\right)\right)\hat{p}_{s}^{\pi_{U}}\left(y\left|x\right.\right)\eta\left(y\right)\mathsf{m}\left(dy\right)ds.
\end{alignat*}
If we consider $G_{\frac{1}{n}}^{\eta}\eta$ as $\sqrt{\pi_{n}}$
this is the desired sequence of proper priors. In other words, we
have the following theorem: \begin{theorem} \label{prop:Diri_0}Assume
that the Dirichlet form, $\left(\mathcal{E},\mathcal{F}\right)$,
is recurrent. $G_{\frac{1}{n}}^{\eta}\eta$ is a function defined
on $\mathbb{R}^{d}$. Then, $G_{\frac{1}{n}}^{\eta}\eta\in\mathcal{F}\subset L^{2}\left(\mathbb{R}^{d};\mathsf{m}\right)$,
and is square-integrable, proper, $0\leqq G_{\frac{1}{n}}^{\eta}\eta\left(x\right)\uparrow1,\ \mathsf{m}\textrm{-a.e.}$,
and asymptotically uniform, non-negative. From this, we have, 
\[
\lim_{n\rightarrow\infty}\mathcal{E}\left(G_{\frac{1}{n}}^{\eta}\eta,G_{\frac{1}{n}}^{\eta}\eta\right)=0.
\]

\end{theorem} \begin{proof} See Supplementary Material Appendix~\ref{app:Diri_0}.
\end{proof} Thus, Corollary~\ref{cor:Cauchy} is proven.

To show the admissibility of the Cauchy distribution when $d=1$,
we need to show the existence of a sequence of $\hat{p}^{\pi}$ that
converges the Bayes risk difference to zero. This is equivalent to
the existence of a sequence, $M^{\pi}$, that converges the Dirichlet
form to zero. From, Lemma~\ref{lem:Diri<prior}, we only need to
consider the prior sequence. Since the Dirichlet form of the Cauchy
process is recurrent when $d=1$ (from Example~\ref{transiCauchy}),
if we input the functional sequence of a proper prior, $\left\{ G_{\frac{1}{n}}^{\eta}\eta\right\} $,
from Theorem~\ref{prop:Diri_0}, we can show that the Bayes risk
difference converges to zero. This is because the prior sequence
is a measure sequence that approximates the Lebesgue measure by the
$L^{2}\left(\mathbb{R}^{d};\mathsf{m}\right)$ function. Therefore,
for a Cauchy distribution with $d=1$, the uniform Bayes predictive
distribution, $\hat{p}^{\pi_{U}}$, is admissible from Blyth's method
\citep[][Lemma 1.]{Brown-EdGeorge-Xu_08}.

Under a Gaussian distribution, $G_{\frac{1}{n}}^{\eta}\eta$ is equivalent
to the Stein prior, which is a Green function. We can consider $G_{\frac{1}{n}}^{\eta}\eta$
to be an extension of the Stein prior to infinitely divisible distributions.
However, as noted in Section~\ref{sec:2.2}, we are not evaluating
the risk difference (not the Bayes risk), so for the statement \textquotedblright under
a Gaussian distribution, the Bayes predictive distribution with a
Stein prior, which is a Green function, dominates a uniform prior
distribution and is also admissible," we are only extending the \textquotedblright admissible"
part. The rest is for future research.

\section{\label{subsec:Markov-Bayes}Bayesian predictive distributions as
Markov transition probabilities}

In this section, we correspond the predictive distribution, $\hat{p}^{\pi_{U}}\left(y\left|x\right.\right)$,
and the continuous-time Markov process, $\left\{ X_{t}\right\} _{t\geqq0}$.
Thus, we will derive the continuous-time Markov process from the predictive
distribution, $\hat{p}^{\pi_{U}}\left(y\left|x\right.\right)$, and
connect the known results from Markov processes to the statistical
decision problem. Specifically, we define the Dirichlet form that
corresponds to the predictive distribution, $\hat{p}^{\pi_{U}}\left(y\left|x\right.\right)$,
and the concept of transience/recurrence of Markov processes. For
a more detailed explanation regarding Markov processes, see \cite{Fukushima-Oshima-Takeda_10}.

Let the transition probability of a continuous-time Markov process,
$\left\{ X_{t}\right\} _{t\geqq0}$, be $p_{t}\left(\left.y\right|x\right)$.
This can be interpreted as the density function of the probability
that a Markov process, $\left\{ X_{t}\right\} $, with initial value,
$X_{0}=x$, takes the value, $X_{t}=y$, at time, $t$.

Let us reinterpret the likelihood function, $p_{c}\left(x\left|\theta\right.\right)$,
of the symmetric $d$-dimensional Location-Scale family, $p_{c}\left(x\left|\theta\right.\right)$,
as the transition density function of a stochastic process, $\left\{ X_{t}\right\} _{t\geqq0}$,
with initial value, $X_{0}=\theta$, and takes value, $X_{c}=x$,
at time, $c$, and the random variable, $X\sim p_{c}\left(x\left|\theta\right.\right)$,
as a stochastic process that takes value, $X_{c}$, at time, $t=c$.
From symmetry, we have $p_{c}\left(x\left|\theta\right.\right)=p_{c}\left(\theta\left|x\right.\right)$,
thus $p_{c}\left(\theta\left|x\right.\right)$ can be interpreted
as a transition density function of a stochastic process with initial
value, $X_{0}=x$, and takes value, $X_{c}=\theta$, at time, $c$.
Then, from the symmetry of the distribution, $p_{c}\left(x\left|\theta\right.\right)$,
the predictive distribution, $\hat{p}^{\pi_{U}}\left(y\left|x\right.\right)$,
can be interpreted as a convolution distribution of two random variables,
$Y\sim p_{c}\left(y\left|\theta\right.\right)$, $\theta\sim p_{c}\left(x\left|\theta\right.\right)=p_{c}\left(\theta\left|x\right.\right)$:
\[
\hat{p}^{\pi_{U}}\left(y\left|x\right.\right)=p_{2c}\left(\left.y\right|x\right)=\int_{\mathbb{R}^{d}}p_{c}\left(y\left|\theta\right.\right)p_{c}\left(\theta\left|x\right.\right)\mathsf{m}\left(d\theta\right).
\]
Here, $\mathsf{m}\left(d\theta\right)$ is a Lebesgue measure in $\mathbb{R}^{d}$.
$\hat{p}^{\pi_{U}}\left(y\left|x\right.\right)$ can be seen as the
transition density function of the stochastic process, $\left\{ X_{t}\right\} _{t\geqq0}$,
with initial value, $X_{0}=x$, that takes value, $X_{2c}=y$, at
time, $2c$.

In corresponding the transition density function, $p_{t}\left(\left.y\right|x\right)$,
of the continuous-time stochastic process, $\left\{ X_{t}\right\} _{t\geqq0}$,
to the predictive distribution, $\hat{p}^{\pi_{U}}\left(y\left|x\right.\right)$,
we assumed that the transition density function, $p_{2c}\left(\left.y\right|x\right)$,
at time, $t=2c$, is equal to the predictive distribution, $\hat{p}^{\pi_{U}}\left(y\left|x\right.\right)$.
The transition density function, $p_{t}\left(\left.y\right|x\right)$,
can take any positive real value including $t=\left(0,c,2c\right)$,
at time, $t$. However, we assume that the scale parameter
of the symmetric $d$-dimensional Location-Scale family is equal,
thus the transition density function, $p_{t}\left(\left.y\right|x\right)$,
does not depend on $t$. For example, if we set $t=0.31415$, $p_{t}\left(\left.y\right|x\right)$
is a symmetric $d$-dimensional Location-Scale family with location,
$\theta=x$, and scale parameter, $c=0.31415$. Therefore, the uniform
Bayes predictive distribution, $\hat{p}^{\pi_{U}}\left(y\left|x\right.\right)$,
is the transition probability of the convolution semigroup, and is
to simply consider the transition probability, $p_{t}$, to be the
predictive distribution with scale parameter, $c$ and time, $t$.
Thus, constructing a continuous-time transition probability function
from the uniform Bayes predictive distribution is straightforward.

Next, using the transition probability function, $p_{t}\left(\left.y\right|x\right)$,
we define the notion of transience and recurrence of a Markov process,
Define the potential operator as 
\begin{align*}
Rf\left(x\right) & =\int_{0}^{\infty}\left\{ \int_{\mathbb{R}^{d}}p_{s}\left(\left.y\right|x\right)f\left(y\right)\mathsf{m}\left(dy\right)\right\} ds.
\end{align*}
Intuitively, $\int_{\mathbb{R}^{d}}p_{s}\left(\left.y\right|x\right)1_{B}\left(y\right)\mathsf{m}\left(dy\right)$
is the probability that a stochastic process, with initial value,
$x$, arrives at $B$ after time, $s$. To integrate this probability
over time, $\left[0,\infty\right)$, means to obtain the number of
times the stochastic process arrives at $B$ after a long time. Then,
we have the following definition: \begin{definition} \textit{The
transition probability function, $\left\{ p_{t}\right\} _{t\geq0}$,
is transient if there exists positive integrable function, $g\in L_{+}^{1}\left(\mathbb{R}^{d};\mathsf{m}\right)$,
satisfying $\mathsf{m}\left(x\left|g\left(x\right)=0\right.\right)=0$,
such that 
\[
Rg<\infty,\ \mathsf{m}\textrm{-a.e.}
\]
The transition probability function, $\left\{ p_{t}\right\} _{t\geq0}$,
is recurrent if there exists positive integrable function, $g\in L_{+}^{1}\left(\mathbb{R}^{d};\mathsf{m}\right)$,
such that 
\begin{alignat*}{1}
Rg & =\infty\,\textrm{or}\ 0,\ \mathsf{m}\textrm{-a.e.}
\end{alignat*}
} \end{definition} Therefore, whether the integral value of the transition
probability with regard to time is finite or infinite determines if
it is transient or recurrent. As we will see later, whether the integral
value of the transition value with regard to time, $Rg$, is finite
or infinite determines whether Eq.~(\ref{eq:BayesRiskDiffe}) is
non-zero or zero.

Lastly, we will briefly explain the Dirichlet form. For a more rigorous
explanation, see \cite{Fukushima-Oshima-Takeda_10}. Let $L^{2}\left(\mathbb{R}^{d};\mathsf{m}\right)$
be the entire square-integrable function regarding the Lebesgue measure,
$\mathsf{m}$, on $\mathbb{R}^{d}$.

First, there exists a measure, $m\left(\cdot\right)$, that is an
invariant measure 
\[
m\left(dy\right)=\int_{\mathbb{R}^{d}}p_{t}\left(\left.y\right|x\right)m\left(dx\right),\ \left(t>0\right)
\]
with regard to the transition probability, $p_{t}\left(\left.y\right|x\right)$.
The Lebesgue measure $\mathsf{m}$ is always an invariant measure,
though there exists $p_{t}\left(\left.y\right|x\right)$, such that
the invariant measure is a finite measure, as Eq.~(\ref{eq:Marginal}).
However, as we will see later, for the uniform prior predictive distribution,
$\hat{p}^{\pi_{U}}\left(y\left|x\right.\right)$, that we consider,
the invariant distribution of the transition probability, $p_{t}\left(\left.y\right|x\right)$,
is only a Lebesgue measure.

Next, if we differential the transition probability, $p_{t}\left(\left.y\right|x\right)$
regarding time, $t$, we obtain the parabolic partial differential
equation: 
\[
\frac{d}{dt}p_{t}\left(\left.y\right|x\right)=\mathcal{A}\cdot p_{t}\left(\left.y\right|x\right)
\]
regarding the operator, $\mathcal{A}$, that corresponds to $p_{t}$.
The operator, $\mathcal{A}$, differs according to the transition
probability, $p_{t}\left(\left.y\right|x\right)$. Note that $\mathcal{A}\cdot p_{t}\left(\left.y\right|x\right)$
operates on the variable, $x$. The function space that is the domain
for the operator, $\mathcal{A}$, is the linear subspace, $\mathcal{F}\subset L^{2}\left(\mathbb{R}^{d};\mathsf{m}\right)$,
of $L^{2}\left(\mathbb{R}^{d};\mathsf{m}\right)$, and represents
all functions, $f$, such that $\mathcal{A}f\in L^{2}\left(\mathbb{R}^{d};\mathsf{m}\right)$.
The size of this function space, $\mathcal{F}$, depends on $\mathcal{A}$.
This invariant measure, $\mathsf{m}$, and the domain, $\mathcal{F}$,
determined by the operator, $\mathcal{A}$, is what determines the
transience/recurrence of a Markov process with transition probability,
$p_{t}\left(\left.y\right|x\right)$, and the admissibility/inadmissibility
of the predictive distribution, $\hat{p}^{\pi_{U}}\left(y\left|x\right.\right)$.

Similar to considering the eigenvalue problem for matrices, we consider
the following quadratic form: 
\[
\left\langle -\mathcal{A}f,g\right\rangle _{\mathsf{m}}=\int_{\mathbb{R}^{d}}\left(-\mathcal{A}f\left(x\right)\right)g\left(x\right)\mathsf{m}\left(dx\right),\ f,g\in\mathcal{F}.
\]
Regarding the operator, $\mathcal{A}$, of the transition probability,
$p_{t}\left(\left.y\right|x\right)$, obtained from convolution, the
quadratic form is symmetric: $\left\langle -\mathcal{A}f,g\right\rangle _{\mathsf{m}}=\left\langle f,-\mathcal{A}g\right\rangle _{\mathsf{m}}$.
This symmetric quadratic form is denoted as $\mathcal{E}\left(f,g\right)$,
and, with the domain, $\mathcal{F}$, of the operator, $\mathcal{A}$,
we have the Dirichlet form, $\left(\mathcal{E},\mathcal{F}\right)$.

Using $\left(\mathcal{E},\mathcal{F}\right)$, we can extract from
the transition probability, $p_{t}\left(\left.y\right|x\right)$,
the analytic information regarding $\mathsf{m},\mathcal{A},\mathcal{F}$,
and the related information on the Markov process \citep{Fukushima-Oshima-Takeda_10}.
In this paper, we directly connect the Bayes risk difference in Theorem~\ref{prop:BayesRisk}
and $\left(\mathcal{E},\mathcal{F}\right)$. By doing this, we can
apply known properties of Dirichlet form and Markov processes to the
statistical decision problem.

\begin{example} \textit{(Normal distribution and Brownian motion).
Let $X,Y$ be the $d$-dimensional normal random variable with mean,
$\theta$, and covariance, $A$: 
\[
X\sim\mathcal{N}\left(\theta,A\right)=p_{A}\left(x\left|\theta\right.\right),\ Y\sim\mathcal{N}\left(\theta,A\right)=p_{A}\left(y\left|\theta\right.\right).
\]
When observing a datum, $X=x$, the MLE is $\hat{\theta}=x$, and
the Bayes predictive distribution under the uniform prior, $\pi_{U}\left(\theta\right)=1$,
is 
\[
\hat{p}^{\pi_{U}}\left(y\left|x\right.\right)=\frac{1}{\sqrt{\left(2\pi\right)^{d}}}\frac{1}{\sqrt{\det\left(2A\right)}}\exp\left(-\frac{1}{2}\left\langle y-x,\left(2A\right)^{-1}\left(y-x\right)\right\rangle \right),
\]
and is a normal distribution with location $x$: $\mathcal{N}\left(x,2A\right)$.
However, the scale is $2A$ and is not the same as the plug-in predictive
distribution. Setting the corresponding transition probability of
the Markov process, $p_{t}\left(\left.y\right|x\right)$, as 
\[
p_{t}\left(y\left|x\right.\right)=\mathcal{N}\left(x,tA\right),
\]
this is the transition probability function of a Brownian motion with
initial value, $x$, at time, $t=0$. When $t=2$, this is equivalent
to the Bayes predictive distribution. }

\textit{A $d$-Brownian motion is recurrent when $d=1,2$, and is
transient otherwise. In fact, let the transition function $\left\{ p_{t}\left(\left.x\right|0\right);t>0\right\} $
be 
\begin{alignat*}{1}
p_{t}\left(\left.x\right|0\right) & =\frac{1}{\left(2\pi t\right)^{d/2}}\exp\left(-\frac{\left\Vert x\right\Vert ^{2}}{2t}\right)=\prod_{i=1}^{d}\frac{1}{\sqrt{2\pi t}}\exp\left(-\frac{x_{i}^{2}}{2t}\right).
\end{alignat*}
Integrating $p_{t}\left(\left.x\right|0\right)$ with respect to $t$
provides 
\begin{alignat*}{1}
\int_{0}^{\infty}p_{t}\left(\left.x\right|0\right)dt & =\begin{cases}
\infty, & d=1,2,\\
\frac{1}{2}\frac{1}{\pi^{d/2}}\Gamma\left(\frac{d}{2}-1\right)\left\Vert x\right\Vert ^{2-d}, & d\geq3.
\end{cases}
\end{alignat*}
}

\textit{Since $\int_{\mathbb{R}^{d}}p_{t}\left(y\left|x\right.\right)\mathsf{m}\left(dx\right)=1$,
the Lebesgue measure is an invariant measure. If we differentiate,
$p_{t}$, with regard to, $t$, we have 
\[
\frac{d}{dt}p_{t}\left(y\left|x\right.\right)=\frac{1}{2}\sum_{p,q=1}^{d}a_{pq}\frac{\partial}{\partial x_{p}}\frac{\partial}{\partial x_{q}}p_{t}\left(y\left|x\right.\right).
\]
Here, we have $A=\left(a_{pq}\right)$ (matrix $A$, where the $p,q$
elements are $a_{pq}$), and since $A$ is a symmetric matrix, we
have $a_{pq}=a_{qp}$. From this, the quadratic form is symmetric,
\[
\left\langle -\frac{1}{2}\sum_{p,q=1}^{d}a_{pq}\frac{\partial}{\partial x_{p}}\frac{\partial}{\partial x_{q}}f,g\right\rangle _{\mathsf{m}}=\left\langle f,-\frac{1}{2}\sum_{p,q=1}^{d}a_{pq}\frac{\partial}{\partial x_{p}}\frac{\partial}{\partial x_{q}}g\right\rangle _{\mathsf{m}}.
\]
Next, regarding the domain, $\mathcal{F}$, of the operator, $\frac{1}{2}\sum_{p,q=1}^{d}a_{pq}\frac{\partial}{\partial x_{p}}\frac{\partial}{\partial x_{q}}$,
it is immediate that $C_{0}^{\infty}\left(\mathbb{R}^{d}\right)$
(an infinitely continuous differentiable, compact support, continuous
function space) satisfies the conditions. However, it also seems that
a larger function space, e.g., a Sobolev space, $H^{1}\left(\mathbb{R}^{d}\right)$
when $A$ is an identity matrix. The question of what is the largest
space for $\mathcal{F}$ is a difficult one, though it is not necessary
for the main results of this paper, so we leave $\mathcal{F}$ unspecified.
The Dirichlet form is 
\[
\mathcal{E}\left(f,g\right)=\left\langle -\frac{1}{2}\sum_{p,q=1}^{d}a_{pq}\frac{\partial}{\partial x_{p}}\frac{\partial}{\partial x_{q}}f,g\right\rangle _{\mathsf{m}},\ \mathcal{F}=C_{0}^{\infty}\left(\mathbb{R}^{d}\right).
\]
}\end{example}

\begin{example} \textit{(Cauchy distribution and Cauchy processes).
Let}

\textit{ 
\[
X\sim\mathcal{C}\left(\theta,cI\right)=p_{c}\left(x\left|\theta\right.\right),\ Y\sim\mathcal{C}\left(\theta,cI\right)=p_{c}\left(y\left|\theta\right.\right)
\]
be independent $d$-dimensional multivariate Cauchy vectors with common
unknown location $\theta$, and let $p_{c}\left(x\left|\theta\right.\right)$
and $p_{c}\left(y\left|\theta\right.\right)$ denote the conditional
densities of $X$ and $Y$. The scale parameter $cI$ is known. When
observing a datum, $X=x$, the MLE is $\hat{\theta}=x$, and there
are no moments. The Bayes predictive distribution under a uniform
prior, $\pi_{U}\left(\theta\right)$, is 
\[
\hat{p}^{\pi_{U}}\left(y\left|x\right.\right)=\frac{1}{\sqrt{\pi^{d+1}}}\Gamma\left(\frac{d+1}{2}\right)\frac{c+c}{\sqrt{\left(\left\Vert y-x\right\Vert ^{2}+\left(c+c\right)^{2}\right)^{d+1}}}
\]
with its characteristic function $\varphi\left(z\right)=\exp\left(-2c\left\Vert z\right\Vert +i\left\langle x,z\right\rangle \right)$,
which is a Cauchy distribution with location, $x$, and scale $2c$,
the latter making it different from the plug-in predictive distribution:
\[
\hat{p}^{\pi_{U}}\left(y\left|x\right.\right)=\mathcal{C}\left(x,2cI\right).
\]
}

\textit{\label{transiCauchy} The potential density of $\left\{ p_{t}\left(\left.x\right|\theta\right);t>0\right\} $,
which has a Cauchy distribution, $\mathcal{C}\left(\theta,t\right)$,
as its transition probability is 
\[
\begin{aligned}\int_{0}^{\infty}p_{t}\left(x|\theta\right)dt & =\frac{1}{\pi^{\left(d+1\right)/2}}\varGamma\left(\frac{\left(d+1\right)}{2}\right)\int_{0}^{\infty}\frac{t}{\left(\left\Vert x-\theta\right\Vert ^{2}+t^{2}\right)^{\left(d+1\right)/2}}dt\\
 & =\begin{cases}
\infty & \left(d=1\right),\\
\frac{1}{2}\frac{1}{\pi^{\left(d+1\right)/2}}\varGamma\left(\frac{\left(d-1\right)}{2}\right)\left\Vert x\right\Vert ^{1-d} & \left(d\geq2\right).
\end{cases}
\end{aligned}
\]
Hence, the Cauchy process is recurrent for $d=1$ and transient for
$d\geq2$.}

\textit{The generator of the Cauchy process, $\mathcal{A}$, cannot
be represented using differential operators like the Brownian motion
can. For example, the generator for the 1-dimensional Cauchy process
can be expressed with integral operators, 
\begin{alignat*}{1}
\mathcal{A}f\left(x\right)= & \frac{1}{\pi}\int_{-\infty}^{\infty}\left\{ f\left(x+y\right)-f\left(x\right)\right\} \frac{1}{y^{2}}dy.
\end{alignat*}
Here, $f\left(x\right)$ is taken from the subspace of $L^{2}\left(\mathbb{R}^{d};\mathsf{m}\right)$,
such that $\mathcal{A}f\in L^{2}\left(\mathbb{R}^{d};\mathsf{m}\right)$.
Using Perseval's identity of the Fourier transform and the characteristic
function, the Dirichlet form can be written as, 
\begin{align*}
 & \mathcal{F}=\left\{ \left.f\in L^{2}\left(\mathbb{R}^{d};\mathsf{m}\right)\right|\int_{\mathbb{R}^{d}}\left|\hat{f}\left(z\right)\right|^{2}\left\Vert z\right\Vert dz<\infty\right\} ,\\
 & \mathcal{E}\left(f,g\right)=\int_{\mathbb{R}^{d}}\hat{f}\left(z\right)\hat{g}\left(z\right)\left\Vert z\right\Vert dz,\quad f,g\in\mathcal{F},
\end{align*}
where $\hat{f},\hat{g}$ are the Fourier transform of $f,g$ and $\left\Vert z\right\Vert $
is the Euclidean norm of $\mathbb{R}^{d}$. For further details, see
Appendix A.} \end{example}

\begin{example} \label{transiStable} \textit{(1-dimensional stable
distribution and stable processes). Let $\theta=\gamma$, and consider
an additional parameter $\alpha$. We consider the characteristic
function of a 1-dimensional symmetric stable distribution, $\textrm{Stable}\left(\alpha,\gamma,c\right)$,
with parameters, $\left(\alpha,\gamma,c\right)$: 
\[
\mathbb{E}\left[e^{izX}\right]=\exp\left(-c\left\Vert z\right\Vert ^{\alpha}+i\gamma z\right),\ \left(0<\alpha<2\right).
\]
We particularly consider when $0<\alpha\leqq1$. The statistical problem
is in estimating $\gamma$, and consider $\alpha,c$ to be known.
The random variable with this characteristic function is symmetric
around $\gamma$. Further, when $\alpha\leqq1$, $\mathbb{E}\left[X\right]=\infty$
does not have any moments. When $\alpha=1$, the density can be expressed,
using an elementary function, with a Cauchy distribution, $\mathcal{C}\left(\gamma,c\right)$,
though a density function that can be expressed using an elementary
function is not known when $0<\alpha<1$. We observe a datum, 
\[
x=X\sim\textrm{Stable}\left(\alpha,\gamma,c\right)
\]
and construct a Bayes predictive distribution, $p^{\pi_{U}}\left(y\left|x\right.\right)$,
with uniform prior, $\pi\left(\gamma\right)=1$. The characteristic
function, 
\[
\mathbb{E}\left[e^{izY}\right]=\exp\left(-2c\left\Vert z\right\Vert ^{\alpha}+ixz\right),
\]
is then, 
\[
\hat{p}^{\pi_{U}}\left(y\left|x\right.\right)=\textrm{Stable}\left(\alpha,x,2c\right).
\]
}

\textit{Consider the potential operator of a Markov process with a
1-dimensional stable distribution, $\textrm{Stable}\left(\alpha,0,ct\right)$
(without loss of generality, we set the location as $\gamma=0$),
as its transition probability. If $g\left(x\right)\in L^{1}\left(\mathbb{R}^{d};\mathsf{m}\right)$
and its Fourier transform, $\hat{g}\left(x\right)$, are both integrable
and $g\left(x\right)$ is continuous (therefore bounded), we have
the expression \citep{Sato_99}, 
\[
Rg\left(x\right)=\frac{1}{\left(2\pi\right)^{d}}\int_{\mathbb{R}^{d}}e^{i\left\langle z,x\right\rangle }\frac{1}{-\log\left(\mathbb{E}\left[e^{izX}\right]\right)}\hat{g}\left(-z\right)dz.
\]
Here, the real part of the characteristic function is, 
\[
\mathfrak{Re}\left(-\log\left(\mathbb{E}\left[e^{izX}\right]\right)\right)=c\left\Vert z\right\Vert ^{\alpha}
\]
thus, when $0<\alpha<1$, we have 
\[
\int_{\mathbb{R}^{d}}e^{i\left\langle z,x\right\rangle }\frac{1}{c\left\Vert z\right\Vert ^{\alpha}}dz<\infty,
\]
which is transient.}

\textit{Using Perseval's identity for the Fourier transform and the
characteristic function, as with the Cauchy process, the Dirichlet
form can be written as, 
\begin{align*}
 & \mathcal{F}=\left\{ \left.f\in L^{2}\left(\mathbb{R}^{d};\mathsf{m}\right)\right|\int_{\mathbb{R}^{d}}\left|\hat{f}\left(z\right)\right|^{2}\left\Vert z\right\Vert ^{\alpha}dz<\infty\right\} ,\\
 & \mathcal{E}\left(f,g\right)=\int_{\mathbb{R}^{d}}\hat{f}\left(z\right)\hat{g}\left(z\right)\left\Vert z\right\Vert ^{\alpha}dz,\quad f,g\in\mathcal{F}.
\end{align*}
For further details, see Appendix A.1.} \end{example}

The Bayes predictive distribution, $\hat{p}^{\pi}\left(y\left|x\right.\right)$,
with a proper prior, derives a continuous-time Markov chain with initial
value, $x$, and takes the value, $y$, at time, $t=1$. However,
the analytic expression of the continuous-time, Markovian semigroup
is not as easy as $\hat{p}^{\pi_{U}}$, and the expression of the
cylindrical measure is difficult to obtain as well. Although, it is
known that the stationary probability measure is $M^{\pi}\left(x;c\right)$
and induces a stationary Markov process. Here, the initial value is
$x$, and the stationary distribution follows the marginal likelihood,
$M^{\pi}\left(x;c\right)$. Denote the cylindrical measure as $\mathbb{Q}_{x}$.
Because making $M^{\pi}\left(x;c\right)$ a stationary distribution
is not sufficient to uniquely determine the stationary Markov process,
there are multiple choices regarding the cylindrical measure, $\mathbb{Q}_{x}$.
The least we can say is that the Bayes predictive distribution, $\hat{p}^{\pi}\left(y\left|x\right.\right)$,
with a proper prior, is a marginal distribution of a cylindrical measure,
$\mathbb{Q}_{x}$, at time, $t=\left\{ 0,1\right\} $.

The necessary and sufficient condition for admissibility, as in Theorem~\ref{thm.Dirichlet-Recu},
is equivalent to the condition for recurrence of the Dirichlet form.
Conversely, inadmissibility is in an equivalent relation to the transience
of the Dirichlet form (Theorem~\ref{thm:Dirichlet-Tran}).

\section{Variational formula of Kullback-Leibler}

This section will present the main result of this paper. In this section,
we show Eq.~\eqref{eq:BRD=00003D00003D00003DDirichlet}: 
\[
\int_{\mathbb{R}^{d}}\mathsf{KL}\left(\hat{p}^{\pi}\left|\hat{p}^{\pi_{U}}\right.\right)M^{\pi}\left(x;c\right)dx=\mathcal{E}\left(\sqrt{M^{\pi}},\sqrt{M^{\pi}}\right).
\]

\cite{Brown_71} also shows, under quadratic loss, through direct
transformation of the Bayes risk difference, the Dirichlet form. \cite{Eaton_92}
derives an inequality that bounds the Bayes risk difference by the
Dirichlet form of the Markov process. For their derivations, the loss
function and Dirichlet form are directly connected, making it elementary.
This is because the Dirichlet form is in quadratic form, and within
quadratic loss, \cite{Brown_71} derives the generator of the Brownian
motion from Stein's equality, and \cite{Eaton_92} directly derives
the Dirichlet form by looking at the Markov transition probability
as the posterior distribution. On the other hand, this paper's derivation
is by corresponding the predictive distribution and transition probability,
where the correspondence between the KL risk and Dirichlet form is
done by analyzing the Markovian semigroup.

The proof strategy is to: 
\begin{enumerate}
\item Show that the Bayes risk difference of the KL loss, $\int_{\mathbb{R}^{d}}\mathsf{KL}\left(\hat{p}^{\pi}\left|\hat{p}^{\pi_{U}}\right.\right)M^{\pi}\left(x;c\right)dx$,
and the rate function, $I\left(M^{\pi}\right)$ (defined below), is
equal; 
\item Show that the rate function, $I\left(M^{\pi}\right)$, and the Dirichlet
form, $\mathcal{E}\left(\sqrt{M^{\pi}},\sqrt{M^{\pi}}\right)$, is
equal. 
\end{enumerate}
Now, denote the Markovian semigroup under $\hat{p}^{\pi_{U}}$ or
$\mathbb{P}_{x}$ as $\left\{ T_{t}\right\} $, its infinitesimal
generator, $\mathcal{A}$, and the functional space that defines $\mathcal{A}$
as $\mathcal{D}\left(\mathcal{A}\right)$. For $g\in\mathcal{D}\left(\mathcal{A}\right)$,
the functional, $\varphi$, is 
\[
\varphi\left(h,g,\varepsilon\right)=\int_{\mathbb{R}^{d}}\log\frac{g\left(x\right)+\varepsilon}{\left(T_{h}g\right)\left(x\right)+\varepsilon}M^{\pi}\left(x;c\right)dx.
\]
This functional, $\varphi$, represents the expected log likelihood
ratio between the initial distribution, $g\left(x\right)$, and the
Markov process after infinitesimal time, $h$. Here, $g$ is interpreted
as the parameter in the likelihood ratio, and the domain of definition
is $g\in\mathcal{D}\left(\mathcal{A}\right)\subset L^{2}\left(\mathbb{R}^{d},\mathsf{m}\right)$
($g$ is not necessarily a probability density function).

Then, \begin{lemma}\label{lem:rate-KL} Denote $p_{h}=\hat{p}_{h}^{\pi_{U}}\left(y\left|x\right.\right)$.
Then, 
\begin{equation}
\sup_{g\in\mathcal{D}\left(\mathcal{A}\right),\varepsilon>0}\varphi\left(h,g,\varepsilon\right)=\int_{\mathbb{R}^{d}}\mathsf{KL}\left(\hat{p}^{\pi}\left|p_{h}\right.\right)M^{\pi}\left(x;c\right)dx\label{eq:rate-KL}
\end{equation}
holds. Note that $h$ need not be infinitesimal time for this lemma
to hold.\end{lemma}

\begin{proof}

See Supplementary Material Appendix~\ref{app:Proof-of-Lemma_rate-KL}.

\end{proof}

Simply, this lemma states that if the expectation is a stationary
distribution, $M^{\pi}\left(x;c\right)$, then the maximum (expected)
log likelihood ratio is equivalent to the KL risk.

Next, we will define the rate functional. Let $\mathcal{B}_{b}^{+}\left(\mathbb{R}^{d}\right)$
be the set of non-negative, Borel measurable functions $u:\mathbb{R}^{d}\mapsto\mathbb{R}_{+}$
on $\mathbb{R}^{d}$, and set $u_{\varepsilon}=u+\varepsilon$ for
$\varepsilon>0$. The functional $I\left(M^{\pi}\right)$ is defined
as 
\[
I\left(M^{\pi}\right)\triangleq\sup_{u\in\mathcal{B}_{b}^{+}\left(\mathbb{R}^{d}\right),\varepsilon>0}\int_{\mathbb{R}^{d}}\frac{\left(-\mathcal{A}u_{\varepsilon}\right)\left(x\right)}{u_{\varepsilon}\left(x\right)}M^{\pi}\left(x;c\right)dx.
\]
$\mathcal{A}$ is the infinitesimal generator of the Markovian semigroup,
$\left\{ T_{t}\right\} $, under $\hat{p}^{\pi_{U}}$ or $\mathbb{P}_{x}$.
This $I$ is referred to as the rate function in the large deviation
literature. Then, we have

\begin{theorem} \label{prop:rate functional} 
\begin{equation}
\lim_{h\downarrow0}\frac{1}{h}\sup_{g\in\mathcal{D}\left(\mathcal{A}\right),\varepsilon>0}\varphi\left(h,g,\varepsilon\right)=I\left(M^{\pi}\right).\label{eq:rate functional}
\end{equation}
\begin{proof} See Supplementary Material Appendix~\ref{app: rate functional}.
\end{proof} \end{theorem}

If we integrate both sides of Eq.~(\ref{eq:rate functional}) with
regard to $h$ from $0$ to $1$, from Eq.~(\ref{eq:rate-KL}), we
have, 
\begin{equation}
\int_{\mathbb{R}^{d}}\mathsf{KL}\left(\hat{p}^{\pi}\left|\hat{p}^{\pi_{U}}\right.\right)M^{\pi}\left(x;c\right)dx=\sup_{g\in\mathcal{D}\left(\mathcal{A}\right),\varepsilon>0}\varphi\left(1,g,\varepsilon\right)=I\left(M^{\pi}\right).\label{eq:KL-rate}
\end{equation}
Additionally, for the rate function, $I\left(M^{\pi}\right)$, the
following variational formula holds.

\begin{theorem} \label{prop:rate variational} Let $\left(\mathcal{E},\mathcal{F}\right)$
be the Dirichlet form with measure $\mathbb{P}_{x}$, i.e., a Dirichlet
form that follows a Markov process with transition probability, $\hat{p}_{t}^{\pi_{U}}$.
Then, we have $\sqrt{M^{\pi}}\in\mathcal{F}$ and 
\[
I\left(M^{\pi}\right)=\mathcal{E}\left(\sqrt{M^{\pi}},\sqrt{M^{\pi}}\right)
\]
holds. \begin{proof} See Supplementary Material Appendix~\ref{app:rate variational}.
\end{proof} \end{theorem}

From this, we have shown Theorem~\ref{prop:BayesRisk}.

Finally, we show Corollary~\ref{cor:Cauchy} and Corollary~\ref{cor:Stable}.
For the case when $d=1$ for Corollary~\ref{cor:Cauchy}, we have
already shown in \ref{subsec:Sufficient-conditions}.

We first show for the 1-dimensional symmetric stable distribution,
$\textrm{Stable}\left(\alpha,\gamma,c\right)$, with the exponent
$\alpha,\left(0<\alpha<1\right)$. We can also show the inadmissibility
of the Cauchy distribution for $d\geqq2$ in the same manner. The
Bayes risk difference is 
\[
B_{\mathsf{KL}}\left(\theta,\hat{p}^{\pi_{U}}\right)-B_{\mathsf{KL}}\left(\theta,\hat{p}^{\pi}\right)=\mathcal{E}\left(\sqrt{M^{\pi}},\sqrt{M^{\pi}}\right).
\]
Therefore, for any sequence of proper priors, $\left\{ \pi_{n}\right\} $,
if $\mathcal{E}\left(\sqrt{M^{\pi}},\sqrt{M^{\pi}}\right)$ is bounded
away from zero, $\hat{p}^{\pi_{U}}$ is inadmissible. The root $\sqrt{M^{\pi}}$
of the marginal likelihood satisfies 
\[
\int_{\mathbb{R}^{d}}\left(\sqrt{M^{\pi}\left(x\right)}\right)^{2}dx=\int_{\mathbb{R}^{d}}\int_{\mathbb{R}^{d}}p_{c}\left(x\left|\theta\right.\right)\pi\left(\theta\right)d\theta dx=1,
\]
therefore we have $\sqrt{M_{t}^{\pi}}\in L^{2}\left(\mathbb{R}^{d};\mathsf{m}\right)$.
The reference function $g$ is defined as 
\[
g\left(x\right)=\frac{\sqrt{M^{\pi}\left(x\right)}}{\max\left\{ \int_{0}^{\infty}T_{s}\sqrt{M^{\pi}\left(x\right)}ds,1\right\} }.
\]
Since $\left\{ T_{t}\right\} _{t\geqq0}$ is transient from Example~\ref{transiStable},
we have $\int_{0}^{\infty}T_{s}\sqrt{M^{\pi}\left(x\right)}ds<\infty$.
Thus, from Eq.~(\ref{eq:transiDiri}), we have 
\[
0<\int_{\mathbb{R}^{d}}g\left(x\right)\sqrt{M^{\pi}\left(x\right)}dx\leqq\mathcal{E}\left(\sqrt{M^{\pi}},\sqrt{M^{\pi}}\right),
\]
which completes the proof.

\section{Discussion}

\subsection{Generalization to the infinitely divisible distribution}

We consider the infinitely divisible distribution as a generalization
of distributions, such as the Normal, Poisson, Cauchy/stable, negative
binomial, exponential, and $\varGamma$. For a random variable, $X\in\mathbb{R}^{d}$,
to follow an infinitely divisible distribution, is to say that the
characteristic function can be written as the standard form of the
Lévy distribution: 
\begin{alignat}{1}
\mathbb{E}\left[e^{izX}\right]= & \exp\left[-\frac{1}{2}\left\langle z,Az\right\rangle +\int_{\mathbb{R}^{d}}\left(e^{i\left\langle z,x\right\rangle }-1-\frac{i\left\langle z,x\right\rangle }{1+\left|x\right|^{2}}\right)\nu\left(dx\right)+i\left\langle \gamma,z\right\rangle \right].\label{eq:ch_ID}
\end{alignat}
Here, $A$ is a non-negative symmetric matrix, $\nu$ is measurable
on $\mathbb{R}^{d}$ and satisfies $\nu\left\{ 0\right\} =0$ and
$\int\left(\left|x\right|^{2}\land1\right)\nu\left(dx\right)<\infty$,
and $\gamma\in\mathbb{R}^{d}$. The trio, $\left(A,\nu,\gamma\right)$,
is called the generating element of $\mu$. $\nu$ is called the Lévy
measure of $\mu$. $\gamma$ is equivalent to the location parameter.
For example, the $\mathbb{R}^{d}$ Gaussian distribution is when $\nu=0$,
the compound Poisson distribution, we have $\nu=c\sigma$ and $A=0,\gamma=0$,
and the Poisson distribution on $\mathbb{R}$ is when $A=0,\gamma_{0}=0,\nu=c\delta_{1}$.

Let the data, $X$ be given as 
\[
X=\theta+\xi
\]
with the estimand being $\theta\left(\in\mathbb{R}^{d}\right)$. Here,
$\xi$ is a random variable that follows a symmetrized infinitely
divisible distribution, given as $\xi=\xi^{\prime}-\xi^{\prime\prime}$,
where $\xi^{\prime},\xi^{\prime\prime}$ is independently generated
from the same infinitely divisible distribution. $X$ is symmetric
regarding $\theta$: $\left(X-\theta\right)\overset{\textrm{d}}{=}-\left(X-\theta\right)$.
We only consider the MLE when the predictive distribution is a plug-in
estimator. For the Bayes predictive distribution, even if the prior
is improper, it does not preserve the distributional shape \citep[e.g., the Poisson predictive distribution]{Komaki_04}.
Therefore it is difficult to make claims about families of distributions
in a holistic manner.

Let \eqref{eq:ch_ID} be the characteristic function of the infinitely
divisible function from which $\xi$ follows. Consider out of $\left(A,\nu,\gamma\right)$,
$A,\nu$ to be known, and $\gamma=0$. When there is one observation,
$X=x$, if we plug-in $\hat{\theta}=x$ as the estimate of $\theta$,
the characteristic function of the predictive distribution $\hat{p}\left(y\left|x\right.\right)$
based on the MLE $\hat{\theta}=x$ is 
\begin{alignat*}{1}
\mathbb{E}\left[e^{izY}\right]= & \exp\left[-\frac{1}{2}\left\langle z,Az\right\rangle +\int_{\mathbb{R}^{d}}\left(e^{i\left\langle z,y\right\rangle }-1-\frac{i\left\langle z,y\right\rangle }{1+\left|y\right|^{2}}\right)\nu\left(dy\right)+i\left\langle x,z\right\rangle \right].
\end{alignat*}
The Dirichlet for of $\hat{p}\left(y\left|x\right.\right)$ can be
made explicit using the characteristic function of the symmetric convolution
semigroup.

\begin{theorem}\label{Thm:InfiniteDivisible}

The Dirichlet form associated with the predictive distribution, $\hat{p}\left(y\left|x\right.\right)$,
is 
\begin{align*}
 & \mathcal{F}=\left\{ f\in L^{2}\left(\mathbb{R}^{d};\mathsf{m}\right):\int_{\mathbb{R}^{d}}\left\Vert \hat{f}\left(z\right)\right\Vert ^{2}\psi\left(z\right)dz<\infty\right\} ,\\
 & \mathcal{E}\left(f,g\right)=\int_{\mathbb{R}^{d}}\hat{f}\left(z\right)\hat{g}\left(z\right)\psi\left(z\right)dz,\quad f,g\in\mathcal{F},
\end{align*}
where $\psi\left(z\right)$ is the logarithm of characteristic function
$-\mathbb{E}\left[e^{izY}\right]$ and $\hat{f}\left(z\right),\hat{g}\left(z\right)$
are the Fourier transformation of the integrable function, $f,g$
on $\mathbb{R}^{d}$, respectively.

\begin{proof} See Supplementary Material Appendix~\ref{app:InfiniteDivisible}.
\end{proof}

\end{theorem}

Now, the following theorem is known to test for recurrence/transience
of the Lévy process \citep{Sato_99}: \begin{theorem} When $\psi\left(z\right)=0$,
then $\mathfrak{Re}\left(\frac{1}{-\psi\left(z\right)}\right)=\infty$,
$\frac{1}{\left[\mathfrak{Re}\left(-\psi\left(z\right)\right)\right]}=\infty$.
We then set the bounded open neighborhood of $B$ regarding $0$.
If $\left\{ X_{t}\right\} $ is 
\[
\int_{B}\mathfrak{Re}\left(\frac{1}{-\psi\left(z\right)}\right)dz=\infty
\]
it is recurrent, and if 
\[
\int_{B}\mathfrak{Re}\left(\frac{1}{-\psi\left(z\right)}\right)dz<\infty
\]
it is transient. Note that $\mathfrak{Re}\left(\cdot\right)$ is the
real part of the complex number. \end{theorem}

Therefore, regarding the construction of the predictive distribution
of the location, $\theta$, of the symmetric infinitely divisible
distribution, its in/admissibility can be determined by the integral
value of the logarithm of the characteristic function, $\psi\left(z\right)$,
around the origin, $0\in\mathbb{R}^{d}$. From this condition, it
is known that the Lévy process of $d\geqq3$ is transient, and as
a statistical implication, we have the following result:

\begin{theorem} If $d\geqq3$, then the MLE regarding the location,
$\theta$, plug-in predictive distribution of the symmetric infinitely
divisible distribution on $\mathbb{R}^{d}$ is always inadmissible
under $\mathsf{KL}$ loss. \end{theorem}

\subsection{Domination of $\hat{p}^{\pi}$ over $\hat{p}^{\pi_{U}}$}

In this paper, we investigated the Bayes risk difference, between the Bayesian predictive density with uniform prior, $\hat{p}^{\pi_{U}}\left(y\left|x\right.\right)$, and the Bayesian predictive density with proper prior, $\hat{p}^{\pi}\left(y\left|x\right.\right)$, $B_{\mathsf{KL}}\left(\pi,\hat{p}^{\pi_{U}}\right)-B_{\mathsf{KL}}\left(\pi,\hat{p}^{\pi}\right)$, and found that the difference cannot be zero depending on the dimension, $p$, of the parameter, $\theta\left(\in\mathbb{R}^{p}\right)$.
Since $\hat{p}^{\pi}\left(y\left|x\right.\right)$ is an admissible predictive distribution, from Blyth's method, whether the Bayes risk difference is zero or not determines whether it is admissible or inadmissible.
From both Corollary~\ref{cor:Cauchy} and Corollary~\ref{cor:Stable}, when $p>1$ and $p\geqq1$, respectively, $\hat{p}^{\pi_{U}}\left(y\left|x\right.\right)$ is inadmissible, meaning that there exists a predictive distribution that dominates $\hat{p}^{\pi_{U}}\left(y\left|x\right.\right)$.
For a predictive density, $\hat{p}_{1}$, to dominate $\hat{p}_{2}$, is to say that $R_{\mathsf{KL}}\left(\theta,\hat{p}_{1}\right)\leqq R_{\mathsf{KL}}\left(\theta,\hat{p}_{2}\right)$ for all $\theta$ and with strict for some $\theta$.
Thus, to construct a predictive distribution that dominates $\hat{p}^{\pi_{U}}\left(y\left|x\right.\right)$, we have to evaluate the risk itself, rather than the Bayes risk.
This paper only evaluated the Bayes risk, and has not considered risk that does not integrate over the prior.

One direction of future research is to construct an estimator that dominates $\hat{p}^{\pi_{U}}\left(y\left|x\right.\right)$.
The (not Bayes) risk difference under KL loss is 
\begin{alignat*}{1}
 & R_{\mathsf{KL}}\left(\theta,\hat{p}^{\pi_{U}}\right)-R_{\mathsf{KL}}\left(\theta,\hat{p}^{\pi}\right)\\
= & \int_{\mathbb{R}^{p}}\left\{ \int_{\mathbb{R}^{p}}\log\frac{\hat{p}^{\pi}\left(y\left|x\right.\right)}{\hat{p}^{\pi_{U}}\left(y\left|x\right.\right)}p_{c}\left(y\left|\theta\right.\right)dy\right\} p_{c}\left(x\left|\theta\right.\right)dx.
\end{alignat*}
Let the scale parameter to be known and $c=c_{x}=c_{y}$ and $p_{c}\left(y\left|\theta\right.\right)=p_{c}\left(x\left|\theta\right.\right)$.
From \cite{EdGeorge-Liang-Xu_06} Lemma 2., one approach is to analyze the density ratio, $\frac{\hat{p}^{\pi}\left(y\left|x\right.\right)}{\hat{p}^{\pi_{U}}\left(y\left|x\right.\right)}$.
This is because when you have a normal model, $N\left(\theta,cI\right)$, we have $\frac{\hat{p}^{\pi}\left(y\left|x\right.\right)}{\hat{p}^{\pi_{U}}\left(y\left|x\right.\right)}=\frac{m^{\pi}\left(x;\frac{1}{2}c\right)}{m^{\pi}\left(x;c\right)}$.
However, deriving such an analytical expression is difficult, unless under a normal model.

One viable method is to consider, as we do in this paper, two cylindrical
measures, $\mathbb{P}_{x}$ and $\mathbb{Q}_{x}$, of continuous-time
Markov processes $\left\{ X_{t};X_{0}=x\right\} _{t\in\left[0,2c\right]}$
obeying transition probabilities, $\hat{p}^{\pi_{U}}\left(y\left|x\right.\right)$
and $\hat{p}^{\pi}\left(y\left|x\right.\right)$, respectively, and
define the density ratio process, $\frac{d\mathbb{Q}_{x}}{d\mathbb{P}_{x}}$.
The cylindrical measures, $\mathbb{P}_{x}$ and $\mathbb{Q}_{x}$, are the joint probability measure of the continuous
time Markov process, $\left\{ X_{t}\right\} _{t\in\left[0,2c\right]}$, and if we let $0<t_{1}<\cdots<t_{n}<2c$ be the arbitrary finite partition of $\left[0,2c\right]$, the
density ratio process, $\frac{d\mathbb{Q}_{x}}{d\mathbb{P}_{x}}$, is
\[
\frac{d\mathbb{Q}\left(X_{2c},X_{t_{n}},\cdots,X_{t_{1}},X_{0}=x\right)}{d\mathbb{P}\left(X_{2c},X_{t_{n}},\cdots,X_{t_{1}},X_{0}=x\right)}.
\]
The density ratio, $\frac{\hat{p}^{\pi}\left(y\left|x\right.\right)}{\hat{p}^{\pi_{U}}\left(y\left|x\right.\right)}$,
can be interpreted as 
\[
\frac{\hat{p}^{\pi}\left(y\left|x\right.\right)}{\hat{p}^{\pi_{U}}\left(y\left|x\right.\right)}=\frac{\mathbb{Q}\left(X_{2c}=y,X_{t_{n}}\in\mathbb{R}^{p},\cdots,X_{t_{1}}\in\mathbb{R}^{p},X_{0}=x\right)}{\mathbb{P}\left(X_{2c}=y,X_{t_{n}}\in\mathbb{R}^{p},\cdots,X_{t_{1}}\in\mathbb{R}^{p},X_{0}=x\right)}.
\]
Here, $\frac{d\mathbb{Q}_{x}}{d\mathbb{P}_{x}}$ is a change of measure
for the probability measure, $\mathbb{P}_{x}$ to $\mathbb{Q}_{x}$,
and a transformation of drift from a Markov process $\left\{ X_{t}\right\} $
following $\mathbb{P}_{x}$, to a Markov process following $\mathbb{Q}_{x}$.
If we use martingale theory, it may be possible to obtain an analytical
expression for $\frac{d\mathbb{Q}_{x}}{d\mathbb{P}_{x}}$, similar
to Girsanov's formula. This will be left for future work.

\bibliographystyle{imsart-nameyear}
\bibliography{Matome_2022-08}

\newpage{}

\appendix

{\centerline{\textbf{\Large{}{}{}{}{}Supplementary Material
for Inadmissibility and Transience}{\Large{}{}{}{}{}}} 


\section{Preliminaries}

In this section, we will explain Markov processes in continuous-time,
Dirichlet form, transition probability function, cylindrical measure
(measure of continuous-time stochastic processes), and how to discern
the recurrence and transience of Markov processes through Dirichlet
form. Then, we will explain the location estimation problem and its
correspondence to the Bayes predictive distribution.

\subsection{Markov process and Dirichlet form}

Let $\omega$ be a right continuous left limits path in the path space,
$\mathbb{W}$, which takes values on the state space, $\mathbb{R}^{d}$,
and defined by $T$. $T$ is either the interval, $\left[0,\infty\right)$,
or $\left\{ 0,1,2,\cdots\right\} $. Define the projection, $k_{t}$,
$\theta_{t}$, for $\mathbb{W}\mapsto\mathbb{W}$, as 
\begin{alignat*}{1}
\theta_{t}\left(\omega\right)\left(s\right) & =\omega\left(s+t\right)\\
k_{t}\left(\omega\right)\left(s\right) & =\omega\left(t\wedge s\right).
\end{alignat*}
$\mathbb{W}$ is closed under the projection, $k_{t},\theta_{t}$,
and $\mathscr{F}=\mathscr{B}\left(\mathbb{W}\right)$ is the $\sigma$-algebra
generated from the set 
\[
\mathscr{F}=\left\{ \left.\omega\in\mathbb{W}\right|\omega\left(t\right)\in A\right\} \ \left(t>0,^{\forall}A\in\mathscr{B}\left(\mathbb{R}^{d}\right)\right).
\]
Henceforth, $\omega\left(t\right)$ is denoted as $X_{t}\left(\omega\right)$.
Thus, for each, $t\in T$, $X_{t}$ is the measurable projection from
the measurable space, $\left(\mathbb{W},\mathscr{F}\right)$, to the
measurable space, $\left(\mathbb{R}^{d},\mathscr{B}\left(\mathbb{R}^{d}\right)\right)$;
$X_{t}\in\mathscr{F}\left/\mathscr{B}\left(\mathbb{R}^{d}\right)\right.$.
Further, the filtration, $\mathscr{F}_{t}$, is a partial $\sigma$-algebra
of $\mathscr{F}$, written as 
\[
\mathscr{F}_{t}=\left\{ \left.\omega\in\mathbb{W}\right|k_{t}\left(\omega\right)\in B\right\} \ \left(^{\forall}B\in\mathscr{F}\right).
\]
This is a set of all paths up until $t$ in $\mathscr{F}$.

Consider a family of cylindrical measure, $\left\{ \mathbb{P}_{x}\right\} _{x\in\mathbb{R}^{d}}$,
on $\left(\mathbb{W},\mathscr{F}\right)$, with initial value, $x$,
$\mathbb{P}_{x}\left(X_{0}\left(\omega\right)=x\right)=1$, $\mathbb{P}_{x}\left(B\right)$
is $\mathscr{B}\left(\mathbb{R}^{d}\right)$-measurable for all $x\in\mathbb{R}^{d}$,
and satisfies the Markov property: i.e., for $\mathbb{P}_{x}$, measure
$1$, 
\begin{alignat*}{1}
\mathbb{P}_{x}\left(\left.\theta_{t}\left(\omega\right)\in B\right|\mathscr{F}_{t}\right) & =\mathbb{P}_{X_{t}\left(\omega\right)}\left(B\right),\ B\in\mathscr{F}.
\end{alignat*}
If we take any point, $0<s_{1}<s_{2}<\cdots<s_{n}$, within $T$,
for each $x$, 
\begin{alignat*}{1}
\mathbb{P}_{x}\left(\left.X_{s_{n}+t}\in A\right|X_{s_{1}},\cdots,X_{s_{n}}\right) & =\mathbb{P}_{x}\left(\left.X_{t}\left(\theta_{s_{n}}\right)\in A\right|X_{s_{n}}\right)=\mathbb{P}_{X_{s_{n}}}\left(X_{t}\in A\right)
\end{alignat*}
holds with $\mathbb{P}_{x}$ measure $1$. This shows that $\left\{ X_{t}\right\} _{t\geq0}$
is a Markov process on $\mathbb{R}^{d}$ with an initial degenerate
distribution, $\delta_{x}$, for each $x$, and transition probability,
\[
p_{t}\left(\left.A\right|x\right)=\mathbb{P}_{x}\left(X_{t}\in A\right).
\]

Note that $\mathbb{P}_{x}\left(X_{t}\in A\right)$ is a cylindrical
measure, which is to say, for any point, $0<s_{1}<s_{2}<\cdots<s_{n}<t$
in $T$, we have 
\[
\mathbb{P}_{x}\left(X_{s_{1}}\in\mathbb{R}^{d},X_{s_{2}}\in\mathbb{R}^{d},\cdots,X_{s_{n}}\in\mathbb{R}^{d},X_{t}\in A\right).
\]

The semigroup of the positive linear operator, $\left\{ T_{t}\left|t\geqq0\right.\right\} $,
on the space of all bounded measurable functions, $B\left(\mathbb{R}^{d}\right)$,
on $\mathbb{R}^{d}$, is derived by 
\[
T_{t}f\left(x\right)=\int_{\mathbb{R}^{d}}f\left(y\right)p_{t}\left(\left.dy\right|x\right),\ f\in B\left(\mathbb{R}^{d}\right)
\]
from the transition probability, $p_{t}\left(\left.A\right|x\right)$.
Here, $\left\{ T_{t}\right\} $ satisfies $T_{s+t}=T_{s}T_{t}$ and
$\left(s\geqq0\right)$ on $B\left(\mathbb{R}^{d}\right)$. If we
represent $\mathbb{E}_{x}^{\mathbb{P}}\left[\cdot\right]$ with regard
to the expectation of the cylindrical measure, $\mathbb{P}_{x}$,
we have 
\begin{alignat*}{1}
T_{t}f\left(x\right) & =\mathbb{E}_{x}^{\mathbb{P}}\left[f\left(X_{t}\right)\right],
\end{alignat*}
from which $T_{s+t}=T_{s}T_{t}$ follows. Let $\mathsf{m}$ be a positive
Lebesgue measure on $\mathbb{R}^{d}$. For a Markov process, when
the corresponding semigroup, $\left\{ T_{t}\right\} $, satisfies
\[
\int_{S}T_{t}f\left(x\right)g\left(x\right)\mathsf{m}\left(dx\right)=\int_{S}f\left(x\right)T_{t}g\left(x\right)\mathsf{m}\left(dx\right)
\]
for any $t>0$ and any non-negative measurable function, $f,g$, it
is said to be $\mathsf{m}$-symmetric. A symmetric Markov process
is also called a reversible Markov process. Here, $\left\{ T_{t}\right\} $
is uniquely realized as a strongly continuous contraction semigroup
on the real $L^{2}\left(\mathbb{R}^{d};\mathsf{m}\right)$ space.
$L^{2}\left(\mathbb{R}^{d};\mathsf{m}\right)$ is the entire square
integrable function regarding the Lebesgue measure, $\mathsf{m}$,
on $\mathbb{R}^{d}$. Let $\mathcal{A}$ be the generator of $\left\{ T_{t}\right\} $,
then $\mathcal{A}$ is a negative semidefinite self-adjoint operator,
which defines the symmetric form, $\left(\mathcal{E},\mathcal{F}\right)$,
on $L^{2}\left(\mathbb{R}^{d};\mathsf{m}\right)$: 
\[
\mathcal{E}\left(f,g\right)=\left\langle \sqrt{-\mathcal{A}}f,\sqrt{-\mathcal{A}}g\right\rangle _{\mathsf{m}},\ \mathcal{F}=\mathcal{D}\left(\sqrt{-\mathcal{A}}\right).
\]
$\left\langle ,\right\rangle $ is the internal product of $L^{2}\left(\mathbb{R}^{d};\mathsf{m}\right)$,
and $\mathcal{D}\left(\mathcal{A}\right)$ is the functional space
that defines $\mathcal{A}$. $\left(\mathcal{E},\mathcal{F}\right)$
is called the Dirichlet form of the $\mathsf{m}$-symmetric process.
\begin{example} \textit{A $d$-dimensional Brownian motion is symmetric
with regard to the Lebesgue measure, and the Dirichlet form is given
as 
\[
\mathcal{E}\left(f,g\right)=\frac{1}{2}\int_{\mathbb{R}^{N}}\nabla f\left(x\right)\cdot\nabla g\left(x\right)\mathsf{m}\left(dx\right),\ \mathcal{F}=\left\{ f\in L^{2}\left(\mathbb{R}^{d};\mathsf{m}\right)\left|\frac{\partial f}{\partial x_{i}}\in L^{2}\left(\mathbb{R}^{d};\mathsf{m}\right),1\leqq i\leqq d\right.\right\} ,
\]
though all differential are with regard to generalized functions.
$\mathcal{F}$ is a Soblev space. This Dirichlet form was what was
used in \cite{Brown_71}.} \end{example} \begin{example} \textit{The
Dirichlet form of the Cauchy (stable) process on $\mathbb{R}^{d}$
can be specifically shown using the characteristic function of a symmetric
convolutional semigroup. First, define the Fourier transformation
of the integrable function, $f$, on $\mathbb{R}^{d}$ as 
\[
\hat{f}\left(z\right)=\frac{1}{\left(2\pi\right)^{d/2}}\int_{\mathbb{R}^{d}}e^{i\left\langle z,y\right\rangle }f\left(y\right)dy,\ z\in\mathbb{R}^{d}.
\]
Denote the probability measure of the Cauchy (stable) process on $\mathbb{R}^{d}$
as $\mu$ and the probability measure of the Cauchy (stable) process
at time, $t$, as $\mu_{t}$, then $\mu_{t}$ can be represented as
the $t$-times convolution of $\mu$; $\mu^{t*}$. Then the characteristic
function can be written as $\hat{\mu}_{t}=\hat{\mu}^{t}$. The transition
probability is defined as 
\[
p_{t}\left(\left.B\right|x\right)=\mu^{t*}\left(B-x\right)
\]
and the semigroup, $T_{t}f$, is 
\[
T_{t}f\left(x\right)=\int_{\mathbb{R}^{d}}f\left(x+y\right)\mu_{t}\left(dy\right),\ t>0,x\in\mathbb{R}^{d},f\in L^{2}\left(\mathbb{R}^{d};\mathsf{m}\right),
\]
which is a convolution semigroup. Given the characteristic function,
$\hat{\mu}_{t}$, denote the logarithm, $-\psi$, as $\hat{\mu}_{t}\left(z\right)=e^{-t\psi\left(z\right)}$.
Using Perseval's theorem, 
\[
\left\langle f,g\right\rangle _{\mathsf{m}}=\left\langle \hat{f},\overline{\hat{g}}\right\rangle _{\mathsf{m}},\ f,g\in L^{2}\left(\mathbb{R}^{d};\mathsf{m}\right),
\]
we have $\widehat{T_{t}f}=\hat{\mu}_{t}\cdot\hat{f}$, therefore 
\begin{alignat*}{1}
\mathcal{E}^{\left(t\right)}\left(f,f\right)= & \frac{1}{t}\left\langle f-T_{t}f,f\right\rangle _{\mathsf{m}}\\
= & \frac{1}{t}\int_{\mathbb{R}^{d}}\left(\hat{f}\left(z\right)-\hat{\mu}_{t}\left(z\right)\hat{f}\left(z\right)\right)\bar{\hat{f}}\left(z\right)dz\\
= & \int_{\mathbb{R}^{d}}\left|\hat{f}\left(z\right)\right|^{2}\frac{1-\exp\left(-t\psi\left(z\right)\right)}{t}dz.
\end{alignat*}
When $t\downarrow0$, the last integral is monotonically increasing
and converges to $\int_{\mathbb{R}^{d}}\left|\hat{f}\left(z\right)\right|^{2}\psi\left(z\right)dz$.
Thus, 
\begin{align*}
 & \mathcal{F}=\left\{ \left.f\in L^{2}\left(\mathbb{R}^{d};\mathsf{m}\right)\right|\int_{\mathbb{R}^{d}}\left|\hat{f}\left(z\right)\right|^{2}\psi\left(z\right)dz<\infty\right\} ,\\
 & \mathcal{E}\left(f,g\right)=\int_{\mathbb{R}^{d}}\hat{f}\left(z\right)\hat{g}\left(z\right)\psi\left(z\right)dz,\quad f,g\in\mathcal{F}.
\end{align*}
is the Dirichlet form of the Cauchy process.}

\textit{The characteristic function of the predictive distribution
of a $d$-dimensional Cauchy distribution, $\hat{p}_{t}^{\pi_{U}}\left(y\left|x\right.\right)$,
is 
\[
\hat{\mu}\left(z\right)=\exp\left(-2c\left\Vert z\right\Vert +i\left\langle x,z\right\rangle \right)
\]
and the characteristic function of the predictive distribution of
a stable function with $d=1$, $0<\alpha<2$, exponential $\alpha$,
$\hat{p}_{t}^{\pi_{U}}\left(y\left|x\right.\right)$, is 
\[
\hat{\mu}\left(z\right)=\exp\left(-2c\left|z\right|^{\alpha}+ixz\right).
\]
Thus, substituting $\psi$ to $\mathcal{E}\left(f,g\right)$ for each
characteristic function gives us the Dirichlet form.} \end{example}

\subsection{The recurrence and transience of Markov processes and its distinction
using the Dirichlet form}

In this section we introduce the notion of transience and recurrence
of the Markovian semigroup. The recurrence and transience of a Markovian
semigroup can be characterized in terms of its associated Dirichlet
form and extended Dirichlet space.

A $\sigma$-finite measure $\mathsf{m}$ on $\mathbb{R}^{d}$ is called
an invariant measure if 
\begin{alignat*}{1}
\int_{\mathbb{R}^{d}}p_{t}\left(\left.A\right|x\right)\mathsf{m}\left(dx\right) & =\mathsf{m}\left(A\right),\quad\forall t>0,\quad{\forall}A\in\mathscr{B}\left(\mathbb{R}^{d}\right)
\end{alignat*}
holds. A spatially homogeneous Cauchy process has its invariant measure
as a Lebesgue measure. Define the potential operator as 
\begin{align*}
Rf\left(x\right) & =\int_{0}^{\infty}T_{s}f\left(x\right)ds.
\end{align*}
It holds that $Rf\left(x\right)=\lim_{\alpha\rightarrow0}G_{\alpha}f\left(x\right)$.
\begin{definition} \textit{The Markovian semigroup $\left\{ T_{t}\right\} _{t\geq0}$
is transient if there exists a positive integrable function $g\in L_{+}^{1}\left(\mathbb{R}^{d};\mathsf{m}\right)$,
satisfying $\mathsf{m}\left(x\left|g\left(x\right)=0\right.\right)=0$,
such that 
\[
Rg<\infty,\ \mathsf{m}\textrm{-a.e.}
\]
The Markovian semigroup $\left\{ T_{t}\right\} _{t\geq0}$ is recurrent
if there exists a positive integrable function $f\in L_{+}^{1}\left(\mathbb{R}^{d};\mathsf{m}\right)$,
such that 
\begin{alignat*}{1}
Rg & =\infty\,\textrm{or}\ 0,\ \mathsf{m}\textrm{-a.e.}
\end{alignat*}
} \end{definition}

Next, we consider the characterization of recurrence and transience
using the Dirichlet form. First, for transience, the necessary and
sufficient condition is for the Dirichlet form, $\mathcal{E}$, to
be bounded away from 0.

\begin{theorem} \label{thm:Dirichlet-Tran}(Transience of Dirichlet
form). Let $\left\{ T_{t}\right\} _{t>0}$ be a strongly continuous
Markovian semigroup on $L^{2}\left(\mathbb{R}^{d};\mathsf{m}\right)$
and $\left(\mathcal{E},\mathcal{F}\right)$ be the associated Dirichlet
space relative to $L^{2}\left(\mathbb{R}^{d};\mathsf{m}\right)$.
Then, $\left\{ T_{t}\right\} _{t>0}$ is transient if and only if
there exists a bounded, $\mathsf{m}$-integrable, strictly positive
function, $g$, such that $g>0,\ \mathsf{m}\textrm{-a.e.}$ on $\mathbb{R}^{d}$
satisfying 
\begin{alignat*}{2}
\int_{\mathbb{R}^{d}}\left|f\right|gd\mathsf{m}\leqq & \mathcal{E}\left\langle f,f\right\rangle , & ^{\forall}f\in\mathcal{F}.
\end{alignat*}
\end{theorem}

Finding the function, $g$, can be done as follows. First, there exists
a $g\in L^{1}\left(\mathbb{R}^{d};\mathsf{m}\right)\cap L^{2}\left(\mathbb{R}^{d};\mathsf{m}\right)$
that satisfies the following lemma.

\begin{lemma} For any non-negative, integrable function, $g\in L^{1}\left(\mathbb{R}^{d};\mathsf{m}\right)\cap L^{2}\left(\mathbb{R}^{d};\mathsf{m}\right)$,
\[
\sup_{u\in\mathcal{F}}\frac{\left\langle \left|u\right|,g\right\rangle _{\mathsf{m}}}{\sqrt{\mathcal{E}\left(u,u\right)}}=\sqrt{\int_{\mathbb{R}^{d}}g\cdot Ggd\mathsf{m}},
\]
holds. \end{lemma}

Any strongly continuous transient Markovian semigroup $\left\{ T_{t}\right\} _{t>0}$
admits a strictly positive bounded $\mathsf{m}$-integrable function
$g$, such that $\int_{\mathbb{R}^{d}}g\cdot Ggd\mathsf{m}\leq1$.
Such $g$ is called a reference function of the transient semigroup
$\left\{ T_{t}\right\} _{t>0}$. The reference function $g$ can be
constructed by taking a strictly positive bounded measurable function
$f\in L^{1}\left(\mathbb{R}^{d};\mathsf{m}\right)$ with $\int_{\mathbb{R}^{d}}\left|f\left(x\right)\right|\mathsf{m}\left(dx\right)=1$
and letting 
\[
g=\frac{f}{\left(Gf\vee1\right)},
\]
where $g$ is dominated by $f$ and $g>0,\ \mathsf{m}\textrm{-a.e.}$
This $g$ provides 
\[
\int_{\mathbb{R}^{d}}g\cdot Ggd\mathsf{m}\leq\int_{\mathbb{R}^{d}}f\cdot Ggd\mathsf{m}\leq\int_{\mathbb{R}^{d}}Gf\cdot\left(\frac{f}{Gf}\right)d\mathsf{m}=\int_{\mathbb{R}^{d}}fd\mathsf{m}=1.
\]
Therefore, we have, 
\begin{equation}
\frac{\left\langle \left|u\right|,g\right\rangle _{\mathsf{m}}}{\sqrt{\mathcal{E}\left(u,u\right)}}\sqrt{\mathcal{E}\left(u,u\right)}\leqq\sup_{u\in\mathcal{F}}\frac{\left\langle \left|u\right|,g\right\rangle _{\mathsf{m}}}{\sqrt{\mathcal{E}\left(u,u\right)}}\sqrt{\mathcal{E}\left(u,u\right)}\leqq\sqrt{\mathcal{E}\left(u,u\right)},\label{eq:transiDiri}
\end{equation}
where $\sqrt{\mathcal{E}\left(u,u\right)}$ is bounded from below
by a positive number, $\left\langle \left|u\right|,g\right\rangle $,
if $\left\{ T_{t}\right\} _{t>0}$ is transient.

Next, we consider recurrence. This depends on whether there exists
a functional sequence within $\mathcal{F}$ that converges the Dirichlet
form, $\mathcal{E}$, to $0$.

\begin{theorem} \label{thm.Dirichlet-Recu}(Recurrence of Dirichlet
form). For each Dirichlet form $\left(\mathcal{E},\mathcal{F}\right)$
on $L^{2}\left(\mathbb{R}^{d};\mathsf{m}\right)$, the following is
equivalent:\\
 (i) $\left\{ T_{t}\right\} _{t>0}$ is recurrent.\\
 (ii) There exists a sequence $\left\{ f_{n}\right\} $ satisfying
\begin{eqnarray*}
\left\{ f_{n}\right\} \subset\mathcal{F}, & \lim_{n\rightarrow\infty}f_{n}=1\left(\mathsf{m}\textrm{-a.e.}\right), & \lim_{n\rightarrow\infty}\mathcal{E}\left(f_{n},f_{n}\right)=0.
\end{eqnarray*}
(iii) $1\in\mathcal{F}_{e},\mathcal{E}\left(1,1\right)=0$, where
$\mathcal{F}_{e}$ is extended Dirichlet form. In the case where $\mathsf{m}\left(\textrm{supp}\left(p_{t}\left(\cdot\right)\right)\right)<\infty$,
\begin{eqnarray*}
1\in\mathcal{F}, &  & \mathcal{E}\left(1,1\right)=0.
\end{eqnarray*}
\end{theorem}

\section{Proof of Lemma \ref{lem:BGX-Cor1}}

\begin{proof} Extending and modifying the derivation of the Dirichlet
form in \cite{Brown_71} to the case of $X\sim N\left(\mu,vI\right)$,
we have 
\begin{alignat*}{1}
 & \frac{1}{2}\int_{v_{w}}^{v_{x}}\frac{1}{v^{2}}\left[B_{Q}^{v}\left(\pi,\hat{\mu}_{\textrm{MLE}}\right)-B_{Q}^{v}\left(\pi,\hat{\mu}_{\pi}\right)\right]dv\\
 & =\frac{1}{2}\int_{v_{w}}^{v_{x}}\frac{1}{v^{2}}\left[\int_{\mu}\left[R_{Q}^{v}\left(\mu,\hat{\mu}_{\textrm{MLE}}\right)-R_{Q}^{v}\left(\mu,\hat{\mu}_{\pi}\right)\right]\pi\left(\mu\right)d\mu\right]dv\\
 & =\frac{1}{2}\int_{v_{w}}^{v_{x}}\frac{1}{v^{2}}\left[B_{Q}^{v}\left(\mu,\hat{\mu}_{\textrm{MLE}}\right)-B_{Q}^{v}\left(\mu,\hat{\mu}_{\pi}\right)\right]dv\\
 & =\frac{1}{2}\int_{v_{w}}^{v_{x}}\frac{1}{v^{2}}\left[\int_{\mathbb{R}^{p}}\left\Vert v\frac{\nabla m^{\mathsf{m}}\left(z_{v};v\right)}{m^{\mathsf{m}}\left(z_{v};v\right)}-v\frac{\nabla m^{\pi}\left(z_{v};v\right)}{m^{\pi}\left(z_{v};v\right)}\right\Vert ^{2}m_{\pi}\left(z_{v};v\right)dz_{v}\right]dv\\
 & =\frac{1}{2}\int_{v_{w}}^{v_{x}}\left[\int_{\mathbb{R}^{p}}\left\Vert \frac{\nabla m^{\mathsf{m}}\left(z_{v};v\right)}{m^{\mathsf{m}}\left(z_{v};v\right)}-\frac{\nabla m^{\pi}\left(z_{v};v\right)}{m^{\pi}\left(z_{v};v\right)}\right\Vert ^{2}m_{\pi}\left(z_{v};v\right)dz_{v}\right]dv\\
 & =\frac{1}{2}\int_{v_{w}}^{v_{x}}\left[\int_{\mathbb{R}^{p}}\frac{\left\Vert \nabla m^{\pi}\left(z_{v};v\right)\right\Vert ^{2}}{m^{\pi}\left(z_{v};v\right)}dz_{v}\right]dv\\
 & =\frac{1}{2}\int_{v_{w}}^{v_{x}}\left[\int_{\mathbb{R}^{p}}\left\Vert \nabla_{z}2\sqrt{m^{\pi}\left(z_{v};v\right)}\right\Vert ^{2}dz_{v}\right]dv\\
 & =2\int_{v_{w}}^{v_{x}}\mathcal{E}_{\textrm{BM}}\left(\sqrt{m_{v}^{\pi}},\sqrt{m_{v}^{\pi}}\right)dv
\end{alignat*}
and obtain the Dirichlet form for the Brownian motion. Here, $Z$
is dependent on $v$, thus denoted as $Z_{v}$: 
\begin{alignat*}{1}
Z_{v_{w}} & =W=\frac{\left(v_{y}+0v_{x}\right)X+1v_{x}Y}{v_{x}+v_{y}}\sim\mathcal{N}_{p}\left(\mu,v_{w}I\right),\\
Z_{v_{x}} & =X=\frac{\left(v_{y}+v_{x}\right)X+0v_{x}Y}{v_{x}+v_{y}}\sim\mathcal{N}_{p}\left(\mu,v_{w}I\right).
\end{alignat*}

Further, the following holds: 
\[
\int_{v_{w}}^{v_{x}}\mathcal{E}_{\textrm{BM}}\left(\sqrt{m_{v}^{\pi}},\sqrt{m_{v}^{\pi}}\right)dv=v_{x}\mathcal{E}_{\textrm{BM}}\left(\sqrt{m_{v_{x}}^{\pi}},\sqrt{m_{v_{x}}^{\pi}}\right).
\]
Here, the Dirichlet form, $\mathcal{E}_{\textrm{BM}}\left(\cdot,\cdot\right)$,
is the Dirichlet form of the standard Brownian motion. For this, we
use the self-similarity of the Gaussian kernel, 
\[
\frac{1}{\left(2\pi v\right)^{p/2}}\exp\left\{ -\frac{\left\Vert z\right\Vert ^{2}}{2v}\right\} =\lambda^{p}\frac{1}{\left(2\pi v\lambda^{2}\right)^{p/2}}\exp\left\{ -\frac{\left\Vert \lambda z\right\Vert ^{2}}{2v\lambda^{2}}\right\} .
\]
This implies, $\left\{ Z_{v}\right\} \overset{\textrm{d}}{=}\left\{ \frac{1}{\lambda}Z_{v\lambda^{2}}\right\} $
for the Brownian motion. Since the Brownian motion with initial value,
$Z_{0}=\mu$, is $\left\{ Z_{v}-\mu\right\} \overset{\textrm{d}}{=}\left\{ \frac{1}{\lambda}Z_{v\lambda^{2}}-\mu\right\} $,
we can write 
\[
\frac{1}{\left(2\pi v\right)^{p/2}}\exp\left\{ -\frac{\left\Vert z-\mu\right\Vert ^{2}}{2v}\right\} =\lambda^{p}\frac{1}{\left(2\pi v\lambda^{2}\right)^{p/2}}\exp\left\{ -\frac{\left\Vert \lambda z-\lambda\mu\right\Vert ^{2}}{2v\lambda^{2}}\right\} .
\]
Then the marginal likelihood, $m_{\pi}\left(z_{v};v\right)$, and
its derivative function, $\nabla_{z}m_{\pi}\left(z_{v};v\right)$,
is $\mu^{\prime}=\sqrt{\frac{v_{x}}{v}}\mu$, then 
\begin{alignat*}{1}
m_{\pi}\left(z_{v};v\right) & =\int_{\mathbb{R}^{p}}\frac{\left(\sqrt{\frac{v_{x}}{v}}\right)^{p}}{\left(2\pi v\frac{v_{x}}{v}\right)^{p/2}}\exp\left\{ -\frac{\left\Vert \sqrt{\frac{v_{x}}{v}}z-\mu^{\prime}\right\Vert ^{2}}{2v\frac{v_{x}}{v}}\right\} \pi\left(\mu^{\prime}\right)\left(\sqrt{\frac{v}{v_{x}}}\right)^{p}\mathsf{m}\left(d\mu^{\prime}\right)\\
 & =m_{\pi}\left(\sqrt{\frac{v_{x}}{v}}z,v_{x}\right)\\
\nabla_{z}m_{\pi}\left(z_{v};v\right) & =\nabla_{z}m_{\pi}\left(\sqrt{\frac{v_{x}}{v}}z,v_{x}\right)\\
 & =\int_{\mathbb{R}^{p}}\frac{1}{\left(2\pi v_{x}\right)^{p/2}}\exp\left\{ -\frac{\left\Vert \sqrt{\frac{v_{x}}{v}}z-\mu^{\prime}\right\Vert ^{2}}{2v_{x}}\right\} \left\{ -\frac{\sqrt{\frac{v_{x}}{v}}\left\Vert \sqrt{\frac{v_{x}}{v}}z-\mu^{\prime}\right\Vert }{v_{x}}\right\} \pi\left(\mu^{\prime}\right)\mathsf{m}\left(d\mu^{\prime}\right)\\
 & =\sqrt{\frac{v_{x}}{v}}\int_{\mathbb{R}^{p}}\frac{1}{\left(2\pi v_{x}\right)^{p/2}}\exp\left\{ -\frac{\left\Vert \sqrt{\frac{v_{x}}{v}}z-\mu^{\prime}\right\Vert ^{2}}{2v_{x}}\right\} \left\{ -\frac{\left\Vert \sqrt{\frac{v_{x}}{v}}z-\mu^{\prime}\right\Vert }{v_{x}}\right\} \pi\left(\mu^{\prime}\right)\mathsf{m}\left(d\mu^{\prime}\right)\\
 & =\sqrt{\frac{v_{x}}{v}}\frac{\partial m_{\pi}}{\partial z}\left(\sqrt{\frac{v_{x}}{v}}z,v_{x}\right).
\end{alignat*}
Here, let $\frac{\partial m_{\pi}}{\partial z}\left(\sqrt{\frac{v_{x}}{v}}z,v_{x}\right)$
be the derivative function, $\nabla_{z}m_{\pi}\left(z_{v};v\right)$,
where $\left(z_{v},v\right)=\left(\sqrt{\frac{v_{x}}{v}}z,v_{x}\right)$
is substituted for the variable. For the integral, 
\begin{alignat*}{1}
\int_{\mathbb{R}^{p}}\frac{\left\Vert \nabla m^{\pi}\left(z_{v};v\right)\right\Vert ^{2}}{m^{\pi}\left(z_{v};v\right)}\mathsf{m}\left(dz_{v}\right) & =\int_{\mathbb{R}^{p}}\frac{\left\Vert \nabla_{z}m_{\pi}\left(\sqrt{\frac{v_{x}}{v}}z,v_{x}\right)\right\Vert ^{2}}{m_{\pi}\left(\sqrt{\frac{v_{x}}{v}}z,v_{x}\right)}\left(\sqrt{\frac{v_{x}}{v}}\right)^{p}\mathsf{m}\left(dz\right)\\
 & =\left(\sqrt{\frac{v_{x}}{v}}\right)^{p}\int_{\mathbb{R}^{p}}\frac{\frac{v_{x}}{v}\left\Vert \frac{\partial m_{\pi}}{\partial z}\left(\sqrt{\frac{v_{x}}{v}}z,v_{x}\right)\right\Vert ^{2}}{m_{\pi}\left(\sqrt{\frac{v_{x}}{v}}z,v_{x}\right)}\mathsf{m}\left(dz\right)
\end{alignat*}
if we variable transform as $x=\sqrt{\frac{v_{x}}{v}}z$, we have
$\frac{\partial z}{\partial x}=\sqrt{\frac{v}{v_{x}}}$ and $\frac{\partial m_{\pi}}{\partial z}\left(\sqrt{\frac{v_{x}}{v}}z,v_{x}\right)=\sqrt{\frac{v_{x}}{v}}\frac{\partial m_{\pi}}{\partial x}\left(x;v_{x}\right)$,
thus we have 
\begin{alignat*}{1}
\left(\sqrt{\frac{v_{x}}{v}}\right)^{p}\int_{\mathbb{R}^{p}}\frac{\frac{v_{x}}{v}\left\Vert \frac{\partial m_{\pi}}{\partial z}\left(\sqrt{\frac{v_{x}}{v}}z,v_{x}\right)\right\Vert ^{2}}{m_{\pi}\left(\sqrt{\frac{v_{x}}{v}}z,v_{x}\right)}\mathsf{m}\left(dz\right) & =\left(\sqrt{\frac{v_{x}}{v}}\right)^{p}\int_{\mathbb{R}^{p}}\frac{\left(\frac{v_{x}}{v}\right)^{2}\left\Vert \frac{\partial m_{\pi}}{\partial x}\left(x;v_{x}\right)\right\Vert ^{2}}{m_{\pi}\left(x;v_{x}\right)}\left(\sqrt{\frac{v}{v_{x}}}\right)^{p}\mathsf{m}\left(dx\right)\\
 & =\int_{\mathbb{R}^{p}}\frac{\left(\frac{v_{x}}{v}\right)^{2}\left\Vert \frac{\partial m_{\pi}}{\partial x}\left(x;v_{x}\right)\right\Vert ^{2}}{m_{\pi}\left(x;v_{x}\right)}\mathsf{m}\left(dx\right).
\end{alignat*}
Therefore, 
\begin{alignat*}{1}
\frac{1}{2}\int_{v_{w}}^{v_{x}}\left[\int_{\mathbb{R}^{p}}\frac{\left\Vert \nabla m_{\pi}\left(Z_{v};v\right)\right\Vert ^{2}}{m_{\pi}\left(Z_{v};v\right)}\mathsf{m}\left(dz_{v}\right)\right]dv & =\frac{1}{2}\int_{v_{w}}^{v_{x}}\int_{\mathbb{R}^{p}}\frac{\left(\frac{v_{x}}{v}\right)^{2}\left\Vert \frac{\partial m_{\pi}}{\partial x}\left(x;v_{x}\right)\right\Vert ^{2}}{m_{\pi}\left(x;v_{x}\right)}\mathsf{m}\left(dx\right)dv\\
 & =\frac{1}{2}v_{x}\int_{\mathbb{R}^{p}}\frac{\left\Vert \frac{\partial m_{\pi}}{\partial x}\left(x;v_{x}\right)\right\Vert ^{2}}{m_{\pi}\left(x;v_{x}\right)}\mathsf{m}\left(dx\right)\\
 & =2v_{x}\int_{\mathbb{R}^{p}}\left\Vert \frac{\partial}{\partial x}\sqrt{m_{\pi}\left(x;v_{x}\right)}\right\Vert ^{2}\mathsf{m}\left(dx\right).
\end{alignat*}

$2v_{x}\mathcal{E}_{\textrm{BM}}\left(\sqrt{m_{v_{x}}^{\pi}},\sqrt{m_{v_{x}}^{\pi}}\right)$
is 
\begin{alignat*}{1}
2v_{x}\mathcal{E}_{\textrm{BM}}\left(\sqrt{m_{v_{x}}^{\pi}},\sqrt{m_{v_{x}}^{\pi}}\right) & =2v_{x}\int_{\mathbb{R}^{d}}\left(-\nabla^{2}\sqrt{m_{\pi}\left(x;v_{x}\right)}\right)\sqrt{m_{\pi}\left(x;v_{x}\right)}\mathsf{m}\left(dx\right)\\
 & =\int_{\mathbb{R}^{d}}\left(-2v_{x}\nabla^{2}\sqrt{m_{\pi}\left(x;v_{x}\right)}\right)\sqrt{m_{\pi}\left(x;v_{x}\right)}\mathsf{m}\left(dx\right)\\
 & =\mathcal{E}_{\textrm{BM}}^{2v_{x}}\left(\sqrt{m_{v_{x}}^{\pi}},\sqrt{m_{v_{x}}^{\pi}}\right),
\end{alignat*}
so the generator is the Dirichlet form of $2v_{x}\nabla^{2}$. This
is equivalent to the Dirichlet form of the Markov process when the
transition probability is $\hat{p}_{\pi_{U}}$. \end{proof}

\section{Proof of Lemma~\ref{lem:Diri<prior} \label{app:Diri<prior}}

\begin{lemma} Let $\left(\mathcal{E},\mathcal{F}\right)$ be the
Dirichlet form with measure $\mathbb{P}_{x}$, i.e., a Dirichlet form
that follows a Markov process with transition probability, $\hat{p}_{t}^{\pi_{U}}$.
Then, $\sqrt{M^{\pi}}\in\mathcal{F}$, and $\mathcal{E}\left(\sqrt{M^{\pi}},\sqrt{M^{\pi}}\right)$
has the following inequality, 
\[
\mathcal{E}\left(\sqrt{M^{\pi}},\sqrt{M^{\pi}}\right)\leqq\frac{1}{t}\left\{ \left\langle \sqrt{\pi},\sqrt{\pi}\right\rangle _{\mathsf{m}}-\left\langle \sqrt{M^{\pi}},\sqrt{M^{\pi}}\right\rangle _{\mathsf{m}}\right\} \leqq\mathcal{E}\left(\sqrt{\pi},\sqrt{\pi}\right).
\]
\begin{proof} Given, $0\leqq\lambda$ we integrate the inequality,
\[
\lambda\sqrt{e^{-2t\lambda}}\leqq\frac{1}{t}\left(1-\sqrt{e^{-2t\lambda}}\right)\leqq\lambda,
\]
with regard to the measure, $d\left\langle E_{\lambda}\sqrt{\pi},\sqrt{\pi}\right\rangle $,
of the spectral family, $\left\{ E_{\lambda}\right\} _{\lambda\geqq0}$,
that corresponds to $\mathcal{E}$. Then, we have, 
\begin{alignat*}{1}
\mathcal{E}\left(\sqrt{T_{t}\pi},\sqrt{T_{t}\pi}\right) & =\int_{0}^{\infty}\lambda e^{-\lambda t}d\left\langle E_{\lambda}\sqrt{\pi},\sqrt{\pi}\right\rangle \\
 & \leqq\int_{0}^{\infty}\frac{1}{t}\left(1-e^{-t\lambda}\right)d\left\langle E_{\lambda}\sqrt{\pi},\sqrt{\pi}\right\rangle \\
 & \leqq\int_{0}^{\infty}\lambda d\left\langle E_{\lambda}\sqrt{\pi},\sqrt{\pi}\right\rangle =\mathcal{E}\left(\sqrt{\pi},\sqrt{\pi}\right).
\end{alignat*}
Here, we have 
\begin{alignat*}{1}
\int_{0}^{\infty}\frac{1}{t}\left(1-e^{-t\lambda}\right)d\left\langle E_{\lambda}\sqrt{\pi},\sqrt{\pi}\right\rangle  & =\frac{1}{t}\left\{ \left\langle \sqrt{\pi},\sqrt{\pi}\right\rangle _{\mathsf{m}}-\left\langle \sqrt{T_{t}\pi},\sqrt{T_{t}\pi}\right\rangle _{\mathsf{m}}\right\} ,
\end{alignat*}
and 
\[
\sqrt{T_{1}\pi\left(x\right)}=\sqrt{\int_{\varTheta}p_{c}\left(\theta\left|x\right.\right)\pi\left(d\theta\right)}=\sqrt{M^{\pi}\left(x;c\right)}.
\]
\end{proof} \end{lemma}

\section{Proof of Theorem~\ref{prop:Diri_0}\label{app:Diri_0}}

\begin{proof} First, in general, for $f\in L^{2}\left(\mathbb{R}^{d};\mathsf{m}\right)$,
$g\in\mathcal{F}$, $\alpha>0$, if we set 
\[
f_{n}=G_{\frac{1}{n}}^{\eta}\eta,
\]
we have ,$f_{n}\in\mathcal{F}$ and $0\leqq f_{n}\left(x\right)\uparrow1,\left[m\right]$.
This is because, 
\[
\mathcal{E}_{\alpha}\left(G_{\alpha}^{\eta}f,g\right)=\mathcal{E}_{\alpha}^{\eta}\left(G_{\alpha}^{\eta}f,g\right)-\left\langle G_{\alpha}^{\eta}f,g\right\rangle _{\eta\cdot\mathsf{m}}=\left\langle f-\eta G_{\alpha}^{\eta}f,g\right\rangle _{\mathsf{m}}
\]
and 
\[
G_{\alpha}^{\eta}f=G_{\alpha}\left(f-\eta G_{\alpha}^{\eta}f\right),\ \alpha>0.
\]
On the other hand, we have $0\leqq G^{\eta}\eta\leqq1,\ \mathsf{m}\textrm{-a.e.}$.
If we substitute $f$ with $\eta$ in $G_{\alpha}^{\eta}f=G_{\alpha}\left(f-\eta G_{\alpha}^{\eta}f\right)$,
and $\alpha\downarrow0$, then, 
\[
G\eta\left(1-G^{\eta}\eta\right)=G\left(\eta-\eta G^{\eta}\eta\right)=G^{\eta}\eta\leqq1,\ \mathsf{m}\textrm{-a.e.}
\]
Using recurrence, we have $G^{\eta}\eta=1$ from, 
\[
^{\exists}g\in L_{+}^{1}\left(\mathbb{R}^{d};\mathsf{m}\right),\ Gg=\infty,\ \mathsf{m}\textrm{-a.e.}
\]
This is because, since $E=\left\{ x\in E\left|Gg\left(x\right)=\infty\right.\right\} $,
$\mathsf{m}\textrm{-a.e.}$ and $G\eta\left(1-G^{\eta}\eta\right)\leqq1$,
$\mathsf{m}\textrm{-a.e.}$, when $C=\left\{ x\in E\left|G\eta<\infty\right.\right\} $,
we have 
\[
G\eta=0,\ \textrm{and}\ \eta=0,
\]
$\mathsf{m}\textrm{-a.e.}$ on $C$. However, this contradicts $\eta>0,\ \mathsf{m}\textrm{-a.e.}$.
The rest is when $G\eta=\infty$, however, $G\eta\left(1-G^{\eta}\eta\right)\leqq1$
only holds when $G^{\eta}\eta=1$. If we let$f_{n}=G_{\frac{1}{n}}^{\eta}\eta$,
we have $0\leqq f_{n}\uparrow1$ and $n\rightarrow\infty$, then 
\[
\mathcal{E}\left(f_{n},f_{n}\right)\leqq\mathcal{E}_{\frac{1}{n}}\left(f_{n},f_{n}\right)=\left\langle \eta-\eta f_{n},f_{n}\right\rangle _{\mathsf{m}}\leqq\int_{\mathbb{R}^{d}}\left(\eta-\eta f_{n}\right)d\mathsf{m}\rightarrow0.
\]
\end{proof}

\section{Proof of Lemma~\ref{lem:rate-KL}\label{app:Proof-of-Lemma_rate-KL}}

First, with regard to 1., for $\mathsf{KL}\left(\hat{p}^{\pi}\left|\hat{p}^{\pi_{U}}\right.\right)$,
we show the following variational equality.

\begin{theorem} \label{thm:VariFomla}(Donsker-Varadhan variational
formula for the relative entropy). Denote the cylindrical measure
under $\hat{p}^{\pi},\hat{p}^{\pi_{U}}$ as $\mathbb{Q}_{x},\mathbb{P}_{x}$,
and denote the Markov process as $\left\{ X_{t}^{\mathbb{Q}}\right\} ,\left\{ X_{t}^{\mathbb{P}}\right\} $.
Then, the KL divergence, $\mathsf{KL}\left(\hat{p}^{\pi}\left|\hat{p}^{\pi_{U}}\right.\right)$,
is 
\begin{equation}
\mathsf{KL}\left(\hat{p}^{\pi}\left|\hat{p}^{\pi_{U}}\right.\right)=\sup_{g\in\mathcal{B}_{b}\left(\mathbb{R}^{d}\right)}\left\{ \mathbb{E}_{x}^{\mathbb{Q}}\left[g\left(X_{1}^{\mathbb{Q}}\right)\right]-\log\mathbb{E}_{x}^{\mathbb{P}}\left[\exp\left(g\left(X_{1}^{\mathbb{P}}\right)\right)\right]\right\} .\label{eq:DV-VariFml}
\end{equation}
Here, $\mathcal{B}_{b}\left(\mathbb{R}^{d}\right)$ is a set of bounded
measurable function on $\mathbb{R}^{d}$. Further, the above can be
written as 
\[
\mathsf{KL}\left(\hat{p}^{\pi}\left|\hat{p}^{\pi_{U}}\right.\right)=\sup_{g\in C_{b}\left(\mathbb{R}^{d}\right)}\left\{ \mathbb{E}_{x}^{\mathbb{Q}}\left[g\left(X_{1}^{\mathbb{Q}}\right)\right]-\log\mathbb{E}_{x}^{\mathbb{P}}\left[\exp\left(g\left(X_{1}^{\mathbb{P}}\right)\right)\right]\right\} .
\]
Here, $C_{b}\left(\mathbb{R}^{d}\right)$ is a set of bounded measurable
function on $\mathbb{R}^{d}$.

The necessary and sufficient condition for $\mathsf{KL}\left(\hat{p}^{\pi}\left|\hat{p}^{\pi_{U}}\right.\right)$
to be bounded is $\hat{p}^{\pi}\ll\hat{p}^{\pi_{U}}$ and $f\log f\in L^{1}\left(\mathbb{R}^{d};\hat{p}^{\pi_{U}}\right)$
when $\frac{\hat{p}^{\pi}}{\hat{p}^{\pi_{U}}}=f$. When these conditions
are satisfied, we have, 
\[
\mathsf{KL}\left(\hat{p}^{\pi}\left|\hat{p}^{\pi_{U}}\right.\right)=\int_{\mathbb{R}^{d}}\log\frac{\hat{p}^{\pi}}{\hat{p}^{\pi_{U}}}\hat{p}^{\pi}dy=\int_{\mathbb{R}^{d}}f\log f\hat{p}^{\pi_{U}}dy.
\]
\end{theorem}

This theorem is a transformation of the variational formula for the
cross entropy in \cite{Donsker-Varadhan_83}. For the proof, see \cite{Dupuis-Ellis_97}
Lemma1.4.3. C2.

\begin{proof} Let $g_{\varepsilon}=g+\varepsilon$. Then, $\sup_{g\in\mathcal{D}\left(\mathcal{A}\right),\varepsilon>0}\varphi\left(h,g,\varepsilon\right)$
is 
\begin{alignat*}{1}
\sup_{g\in\mathcal{D}\left(\mathcal{A}\right),\varepsilon>0}\int_{X}\log\frac{g_{\varepsilon}\left(x\right)}{\left(T_{h}g_{\varepsilon}\right)\left(x\right)}M^{\pi}\left(dx\right) & =\sup_{g\in\mathcal{D}\left(\mathcal{A}\right),\varepsilon>0}\left\{ \mathbb{E}^{M^{\pi}}\left[\log g_{\varepsilon}\right]-\mathbb{E}^{M^{\pi}}\left[\log T_{h}g_{\varepsilon}\right]\right\} \\
 & =\sup_{g\in\mathcal{D}\left(\mathcal{A}\right),\varepsilon>0}\left\{ \mathbb{E}^{M^{\pi}}\left[\log g_{\varepsilon}\right]-\mathbb{E}^{M^{\pi}}\left[\log\mathbb{E}^{p_{h}}\left[\log e^{g_{\varepsilon}}\right]\right]\right\} \\
 & =\underline{\sup_{\varPhi\in\mathcal{D}_{+}\left(\mathcal{A}\right)}\left\{ \mathbb{E}^{M^{\pi}}\left[\varPhi\right]-\mathbb{E}^{M^{\pi}}\left[\log\mathbb{E}^{p_{h}}\left[e^{\varPhi}\right]\right]\right\} }.
\end{alignat*}
From Theorem~\ref{thm:VariFomla}, we have 
\begin{alignat*}{1}
 & \int_{\mathbb{R}^{d}}\mathsf{KL}\left(\hat{p}^{\pi}\left|p_{h}\right.\right)M^{\pi}\left(x;c\right)dx\\
= & \int_{\mathbb{R}^{d}}\sup_{g\in\mathcal{B}_{b}\left(\mathbb{R}^{d}\right)}\left\{ \mathbb{E}^{\hat{p}^{\pi}}\left[g\right]-\log\mathbb{E}^{p_{h}}\left[e^{g}\right]\right\} M^{\pi}\left(x;c\right)dx.
\end{alignat*}
Because $\hat{p}^{\pi}$ is an $x$ conditional probability of $M^{\pi}$,
\begin{alignat*}{1}
\mathbb{E}^{M^{\pi}}\left[\varPhi\right] & =\mathbb{E}^{M^{\pi}}\left[\mathbb{E}^{\hat{p}^{\pi}}\left[\varPhi\right]\right]\\
 & \leqq\mathbb{E}^{M^{\pi}}\left[\sup_{\varPhi\in\mathcal{B}_{b}\left(\mathbb{R}^{d}\right)}\left\{ \mathbb{E}^{\hat{p}^{\pi}}\left[\varPhi\right]-\log\mathbb{E}^{p_{h}}\left[e^{\varPhi}\right]\right\} \right]+\mathbb{E}^{M^{\pi}}\left[\log\mathbb{E}^{p_{h}}\left[e^{\varPhi}\right]\right]\\
 & =\mathbb{E}^{M^{\pi}}\left[\mathsf{KL}\left(\hat{p}^{\pi}\left|p_{h}\right.\right)\right]+\mathbb{E}^{M^{\pi}}\left[\log\mathbb{E}_{x}^{p_{h}}\left[e^{\varPhi}\right]\right]
\end{alignat*}
therefore, 
\[
\underline{\sup_{\varPhi\in\mathcal{D}_{+}\left(\mathcal{A}\right)}\left\{ \mathbb{E}^{M^{\pi}}\left[\varPhi\right]-\mathbb{E}^{M^{\pi}}\left[\log\mathbb{E}^{p_{h}}\left[e^{\varPhi}\right]\right]\right\} }\leqq\mathbb{E}^{M^{\pi}}\left[\mathsf{KL}\left(\hat{p}^{\pi}\left|p_{h}\right.\right)\right].
\]

Next, we show the following inequality. From Jensen's inequality,
\[
\underline{\left\{ \mathbb{E}^{M^{\pi}}\left[\varPhi\right]-\mathbb{E}^{M^{\pi}}\left[\log\mathbb{E}^{p_{h}}\left[e^{\varPhi}\right]\right]\right\} }\geqq\left\{ \mathbb{E}^{M^{\pi}}\left[\varPhi\right]-\log\mathbb{E}^{M^{\pi}}\mathbb{E}^{p_{h}}\left[e^{\varPhi}\right]\right\} .
\]
If we set $p_{M}\left(y\right)=\int p\left(h,x,y\right)M^{\pi}\left(x;c\right)dx$,
we can rewrite the above as 
\[
\underline{\left\{ \mathbb{E}^{M^{\pi}}\left[\varPhi\right]-\mathbb{E}^{M^{\pi}}\left[\log\mathbb{E}^{p_{h}}\left[e^{\varPhi}\right]\right]\right\} }\geqq\sup_{\varPhi\in\mathcal{B}_{b}\left(\mathbb{R}^{d}\right)}\left\{ \mathbb{E}^{M^{\pi}}\left[\varPhi\right]-\log\mathbb{E}^{p_{M}}\left[e^{\varPhi}\right]\right\} =\mathsf{KL}\left(M^{\pi}\left|p_{M}\right.\right).
\]
Then, from the following Lemma~\ref{lem:4.3}, we have 
\begin{equation}
\mathsf{KL}\left(M^{\pi}\left|p_{M}\right.\right)=\mathbb{E}^{M^{\pi}}\left[\mathsf{KL}\left(M^{\pi}\left(\cdot\left|x\right.\right)\left|p_{M}\left(\cdot\left|x\right.\right)\right.\right)\right].\label{eq:KL_conditional}
\end{equation}
Here, $M^{\pi}\left(\cdot\left|x\right.\right),p_{M}\left(\cdot\left|x\right.\right)$
is the transition probability of a stationary Markov process with
the staring point, $x$, and the stationary distributions are each
$M^{\pi}\left(\cdot\right),p_{M}\left(\cdot\right)$.

Further, since 
\[
\mathbb{E}^{M^{\pi}}\left[\mathsf{KL}\left(M^{\pi}\left(\cdot\left|x\right.\right)\left|p_{M}\left(\cdot\left|x\right.\right)\right.\right)\right]=\mathbb{E}^{M^{\pi}}\left[\mathsf{KL}\left(\hat{p}^{\pi}\left|p_{h}\right.\right)\right],
\]
Therefore, with Eq.~(\ref{eq:KL_conditional}), we show the following
inequality 
\[
\underline{\sup_{\varPhi\in\mathcal{B}_{b}\left(\mathbb{R}^{d}\right)}\left\{ \mathbb{E}^{M^{\pi}}\left[\varPhi\right]-\mathbb{E}^{M^{\pi}}\left[\log\mathbb{E}^{p_{h}}\left[e^{\varPhi}\right]\right]\right\} }\geqq\mathbb{E}^{M^{\pi}}\left[\mathsf{KL}\left(\hat{p}^{\pi}\left|p_{h}\right.\right)\right].
\]
Since $C_{b}\left(\mathbb{R}^{n}\right)\subset\mathcal{D}_{+}\left(\mathcal{A}\right)\subset\mathcal{B}_{b}\left(\mathbb{R}^{d}\right)$,
from Theorem~\ref{thm:VariFomla}, $\sup$ is the same under $\mathcal{D}_{+}\left(\mathcal{A}\right)$
or $\mathcal{B}_{b}\left(\mathbb{R}^{d}\right)$. \end{proof}

\begin{lemma}\label{lem:4.3} 
\begin{alignat*}{1}
\mathsf{KL}\left(M^{\pi}\left|p_{M}\right.\right) & =\mathbb{E}^{M^{\pi}}\left[\mathsf{KL}\left(M^{\pi}\left(\cdot\left|x\right.\right)\left|p_{M}\left(\cdot\left|x\right.\right)\right.\right)\right].
\end{alignat*}
\end{lemma}

\begin{proof} By definition, $\hat{p}^{\pi}$ is the $x$ conditional
probability of $M^{\pi}$. Therefore, 
\[
M^{\pi}\left(y\right)=\int\hat{p}^{\pi}\left(y\left|x\right.\right)M^{\pi}\left(x\right)dx
\]
and, similarly for $p_{M}$, the $x$ conditional distribution is
\[
p_{M}\left(y\right)=\int\hat{p}_{h}\left(y\left|x\right.\right)M^{\pi}\left(x\right)dx.
\]

$\left(\mathbb{R}^{d},\mathscr{B}\left(\mathbb{R}^{d}\right)\right)$
is a Polish space. Let $\mathscr{B}_{1}\subset\mathscr{B}\left(\mathbb{R}^{d}\right)$
be the sub $\sigma$-algebra of $\mathscr{B}\left(\mathbb{R}^{d}\right)$.
The domain of definition of $g$ in \ref{thm:VariFomla} was the measurable
space, $\left(\mathbb{R}^{d},\mathscr{B}\left(\mathbb{R}^{d}\right)\right)$,
but even if we restrict the measurable space to $\left(\mathbb{R}^{d},\mathscr{B}_{1}\right)$,
we can define $\mathsf{KL}$. This restricted $\mathsf{KL}$ is denoted
as $\mathsf{KL}_{\mathscr{B}_{1}}$. Then, we will show that 
\begin{alignat}{1}
\mathsf{KL}\left(M^{\pi}\left|p_{M}\right.\right) & =\mathsf{KL}_{\mathscr{B}_{1}}\left(M^{\pi}\left|p_{M}\right.\right)+\mathbb{E}^{M^{\pi}}\left[\mathsf{KL}\left(M^{\pi}\left(\cdot\left|\mathscr{B}_{1}\right.\right)\left|p_{M}\left(\cdot\left|\mathscr{B}_{1}\right.\right)\right.\right)\right]\label{eq:KL_iteration}
\end{alignat}
holds.

When $\mathsf{KL}_{\mathscr{B}_{1}}\left(M^{\pi}\left|p_{M}\right.\right)=\infty$,
we have, $\mathsf{KL}\geqq\mathsf{KL}_{\mathscr{B}_{1}}$, therefore
$\mathsf{KL}\left(M^{\pi}\left|p_{M}\right.\right)=\infty$.

When $\mathsf{KL}_{\mathscr{B}_{1}}\left(M^{\pi}\left|p_{M}\right.\right)<\infty$,
$\mathsf{KL}\left(M^{\pi}\left|p_{M}\right.\right)<\infty$ does not
necessarily hold, though if we show 
\[
\mathsf{KL}\left(M^{\pi}\left|p_{M}\right.\right)\leqq\mathbb{E}^{M^{\pi}}\left[\mathsf{KL}\left(\hat{p}^{\pi}\left|p_{h}\right.\right)\right]+\mathsf{KL}_{\mathscr{B}_{1}}\left(M^{\pi}\left|p_{M}\right.\right)
\]
when $\mathsf{KL}\left(M^{\pi}\left|p_{M}\right.\right)<\infty$ (given
$\mathsf{KL}\geqq\mathsf{KL}_{\mathscr{B}_{1}}$ it follows automatically
that $\mathsf{KL}_{\mathscr{B}_{1}}\left(M^{\pi}\left|p_{M}\right.\right)<\infty$).
Then, if we show 
\[
\mathsf{KL}\left(M^{\pi}\left|p_{M}\right.\right)=\mathbb{E}^{M^{\pi}}\left[\mathsf{KL}\left(\hat{p}^{\pi}\left|p_{h}\right.\right)\right]+\mathsf{KL}_{\mathscr{B}_{1}}\left(M^{\pi}\left|p_{M}\right.\right)
\]
the proof is complete.

Since we have, 
\begin{alignat*}{1}
\mathbb{E}^{M^{\pi}}\left[\varPhi\right] & =\mathbb{E}^{M^{\pi}}\left[\mathbb{E}^{\hat{p}^{\pi}}\left[\varPhi\right]\right]\\
 & \leqq\mathbb{E}^{M^{\pi}}\left[\sup_{\varPhi\in\mathcal{B}_{b}\left(\mathbb{R}^{d}\right)}\left\{ \mathbb{E}^{\hat{p}^{\pi}}\left[\varPhi\right]-\log\mathbb{E}^{p_{h}}\left[e^{\varPhi}\right]\right\} \right]+\mathbb{E}^{M^{\pi}}\left[\log\mathbb{E}^{p_{h}}\left[e^{\varPhi}\right]\right]\\
 & =\mathbb{E}^{M^{\pi}}\left[\mathsf{KL}\left(\hat{p}^{\pi}\left|p_{h}\right.\right)\right]+\mathbb{E}^{M^{\pi}}\left[\log\mathbb{E}^{p_{h}}\left[e^{\varPhi}\right]\right].
\end{alignat*}
If we set, $\psi=\log\mathbb{E}^{p_{h}}\left[e^{\varPhi}\right]$,
we have $\psi\in\mathscr{B}\left(\left\{ x\right\} \right)$. Then,
\begin{alignat*}{1}
 & \mathbb{E}^{M^{\pi}}\left[\mathsf{KL}\left(\hat{p}^{\pi}\left|p_{h}\right.\right)\right]+\mathbb{E}^{M^{\pi}}\left[\psi\right]\\
\leqq & \mathbb{E}^{M^{\pi}}\left[\mathsf{KL}\left(\hat{p}^{\pi}\left|p_{h}\right.\right)\right]+\mathsf{KL}_{\mathscr{B}_{1}}\left(\hat{p}^{\pi}\left|p_{h}\right.\right)+\log\mathbb{E}^{M^{\pi}}\left[\mathbb{E}^{p_{h}}\left[e^{\varPhi}\right]\right]\\
= & \mathbb{E}^{M^{\pi}}\left[\mathsf{KL}\left(\hat{p}^{\pi}\left|p_{h}\right.\right)\right]+\mathsf{KL}_{\mathscr{B}_{1}}\left(\hat{p}^{\pi}\left|p_{h}\right.\right)+\log\mathbb{E}^{p_{M}}\left[e^{\varPhi}\right].
\end{alignat*}
From this, we have, 
\[
\mathbb{E}^{M^{\pi}}\left[\varPhi\right]-\log\mathbb{E}^{p_{M}}\left[e^{\varPhi}\right]\leqq\mathbb{E}^{M^{\pi}}\left[\mathsf{KL}\left(\hat{p}^{\pi}\left|p_{h}\right.\right)\right]+\mathsf{KL}_{\mathscr{B}_{1}}\left(\hat{p}^{\pi}\left|p_{h}\right.\right),
\]
therefore, 
\[
\sup_{\varPhi}\left\{ \mathbb{E}^{M^{\pi}}\left[\varPhi\right]-\log\mathbb{E}^{p_{M}}\left[e^{\varPhi}\right]\right\} =\mathsf{KL}\left(M^{\pi}\left|p_{M}\right.\right)\leqq\mathbb{E}^{M^{\pi}}\left[\mathsf{KL}\left(\hat{p}^{\pi}\left|p_{h}\right.\right)\right]+\mathsf{KL}_{\mathscr{B}_{1}}\left(\hat{p}^{\pi}\left|p_{h}\right.\right).
\]

Now we show the other direction of the inequality. If $\mathsf{KL}\left(M^{\pi}\left|p_{M}\right.\right)<\infty$,
then $M^{\pi}\ll p_{M}$, therefore we set $g=\frac{dM^{\pi}}{dp_{M}}$.
Similarly, regarding the $x$ conditional probability is $\hat{p}^{\pi}\ll p_{h}$,
therefore 
\[
g\left(\omega\right)=\left.\frac{dM^{\pi}}{dp_{M}}\right|_{\mathscr{B}_{1}}\left(\omega\right)\frac{d\hat{p}^{\pi}\left(\cdot\left|\mathscr{B}_{1}\right.\right)}{dp_{h}\left(\cdot\left|\mathscr{B}_{1}\right.\right)}\left(\omega\right)
\]
holds. Here, $\left.\frac{dM^{\pi}}{dp_{M}}\right|_{\mathscr{B}_{1}}$
is the likelihood ratio, $\frac{dM^{\pi}}{dp_{M}}$, where the domain
of definition is restricted to $\mathscr{B}_{1}$. Then, 
\begin{alignat*}{1}
\mathsf{KL}\left(M^{\pi}\left|p_{M}\right.\right)=\mathbb{E}^{M^{\pi}}\left[\log g\right] & =\mathbb{E}^{M^{\pi}}\left[\log\left.\frac{dM^{\pi}}{dp_{M}}\right|_{\mathscr{B}_{1}}\right]+\mathbb{E}^{M^{\pi}}\left[\log\frac{d\hat{p}^{\pi}\left(\cdot\left|\mathscr{B}_{1}\right.\right)}{dp_{h}\left(\cdot\left|\mathscr{B}_{1}\right.\right)}\right]\\
 & =\mathbb{E}^{M^{\pi}}\left[\log\left.\frac{dM^{\pi}}{dp_{M}}\right|_{\mathscr{B}_{1}}\right]+\mathbb{E}^{M^{\pi}}\left[\mathbb{E}^{p_{h}}\left[\log\frac{d\hat{p}^{\pi}\left(\cdot\left|\mathscr{B}_{1}\right.\right)}{dp_{h}\left(\cdot\left|\mathscr{B}_{1}\right.\right)}\right]\right]\\
 & =\mathbb{E}^{M^{\pi}}\left[\log\left.\frac{dM^{\pi}}{dp_{M}}\right|_{\mathscr{B}_{1}}\right]+\mathbb{E}^{M^{\pi}}\left[\mathsf{KL}\left(\hat{p}^{\pi}\left(\cdot\left|\mathscr{B}_{1}\right.\right)\left|p_{h}\left(\cdot\left|\mathscr{B}_{1}\right.\right)\right.\right)\right]\\
 & =\mathsf{KL}_{\mathscr{B}_{1}}\left(M^{\pi}\left|p_{M}\right.\right)+\mathbb{E}^{M^{\pi}}\left[\mathsf{KL}\left(\hat{p}^{\pi}\left(\cdot\left|\mathscr{B}_{1}\right.\right)\left|p_{h}\left(\cdot\left|\mathscr{B}_{1}\right.\right)\right.\right)\right].
\end{alignat*}
Therefor, we have shown \eqref{thm:VariFomla}.

For \eqref{thm:VariFomla}, if we set $\mathscr{B}_{1}=\left\{ x\right\} $,
we have our final result.

\end{proof}

\section{Proof of Theorem~\ref{prop:rate functional}\label{app: rate functional}}

\begin{lemma}\label{Lemma211} \citep[][Lemma~2.11.]{Jain-Krylov_08}
Let $\mu$ be a probability measure on $\mathbb{R}^{d}$, we assume
$\mu\ll\mathsf{m}$, $h>0$, $u_{\varepsilon}=u+\varepsilon,\left(\varepsilon>0\right)$,
and $u\in\mathcal{B}_{b}^{+}\left(\mathbb{R}^{d}\right)$. Then, 
\[
\inf_{v\in D_{0},\varepsilon>0}\int\log\frac{p_{h}v_{\varepsilon}}{v_{\varepsilon}}d\mu=\inf_{v\in D_{1},\varepsilon>0}\int\log\frac{p_{h}v_{\varepsilon}}{v_{\varepsilon}}d\mu=\inf_{u\in\mathcal{B}_{b}^{+}\left(\mathbb{R}^{d}\right),\varepsilon>0}\int\log\frac{p_{h}u_{\varepsilon}}{u_{\varepsilon}}d\mu.
\]
Where the function spaces are defined as follows: 
\begin{alignat*}{1}
D & =\left\{ u\in\mathcal{B}_{b}^{+}\left(\mathbb{R}^{d}\right)\left|\int_{\mathbb{R}^{d}}ud\mathsf{m}<\infty\right.\right\} \\
D_{0} & =\left\{ v\left|v=\frac{1}{t}\int_{0}^{t}p_{s}uds,\textrm{ for some }u\in D,\textrm{ some }t>0\right.\right\} \\
D_{1} & =\left\{ p_{h}v\left|v\in D_{0},h\geqq0\right.\right\} .
\end{alignat*}
\end{lemma}

\begin{lemma} \citep[][Lemma~2.15.]{Jain-Krylov_08}

When $\mu\ll\mathsf{m}$, 
\[
\lim_{h\downarrow0}\frac{1}{h}\left\{ -\inf_{u\in\mathcal{B}_{b}^{+}\left(\mathbb{R}^{d}\right),\varepsilon>0}\int\log\frac{p_{h}u_{\varepsilon}}{u_{\varepsilon}}d\mu\right\} =I\left(\mu\right).
\]
\end{lemma}

\begin{proof} If $v\in D_{1}$, then there exists $c>0$, such that
\[
\frac{1}{h}\left|p_{h}v-v\right|\leqq c,
\]
for $^{\forall}h>0$. From this if we set $v_{\varepsilon}=v+\varepsilon$,
we have 
\[
\log p_{h}v_{\varepsilon}=\log v_{\varepsilon}+\left(p_{h}v_{\varepsilon}-v_{\varepsilon}\right)\cdot\frac{1}{v_{\varepsilon}}+O\left(h^{2}\right).
\]
Here, $O\left(h^{2}\right)$ is only dependent on $\varepsilon$.
When $h\rightarrow0$, we have $\frac{1}{h}\left(p_{h}v-v\right)\rightarrow\mathcal{A}v$
in terms of $L^{2}\left(\mathbb{R}^{d};\mathsf{m}\right)$, therefore
it is bounded in terms of measure, $\mu\textrm{-a.s.}$. For $^{\forall}v\in D_{1}$,
we have 
\[
\limsup_{h\rightarrow0}\frac{1}{h}\inf_{v\in D_{1},\varepsilon>0}\int\log\frac{p_{h}v_{\varepsilon}}{v_{\varepsilon}}d\mu\leqq\int\frac{\mathcal{A}v}{v_{\varepsilon}}d\mu,
\]
thus, from Lemma~\ref{Lemma211}, we have 
\begin{equation}
\liminf_{h\rightarrow0}\frac{1}{h}\left\{ -\inf_{u\in\mathcal{B}_{b}^{+}\left(\mathbb{R}^{d}\right),\varepsilon>0}\int\log\frac{p_{h}u_{\varepsilon}}{u_{\varepsilon}}d\mu\right\} \geqq I\left(\mu\right).\label{eq:Jain-Krylov_2.17}
\end{equation}

Next, we show the reverse inequality. For $v\in D_{0}$, we set 
\[
\varphi\left(h\right)=\int\log\frac{p_{h}v_{\varepsilon}}{v_{\varepsilon}}d\mu.
\]
Thus, $v_{\varepsilon}=v+\varepsilon$. When $v\in\mathcal{D}\left(\mathcal{A}\right)$,
we have 
\[
\frac{d\varphi}{dh}=\int\frac{\mathcal{A}p_{h}v}{p_{h}v_{\varepsilon}}d\mu\geqq\inf_{v\in D_{1},\varepsilon>0}\int\frac{\mathcal{A}v}{v_{\varepsilon}}d\mu=-I\left(\mu\right).
\]
Regarding $h$, if we integrate from $0$ to $h$, if we use $\varphi\left(0\right)=0$,
we have

\[
\varphi\left(h\right)=\int_{0}^{h}\frac{d\varphi}{dh}dh\geqq=-\int_{0}^{h}I\left(\mu\right)dh=-hI\left(\mu\right),
\]
for $^{\forall}v\in D_{1},\varepsilon>0,h>0$. From \ref{eq:Jain-Krylov_2.17}
and Lemma~\ref{Lemma211}, Q.E.D. \end{proof}

\section{Proof of Theorem~\ref{prop:rate variational}\label{app:rate variational}}

Define $I_{\mathcal{E}}\left(\mu\right)$ as 
\[
I_{\mathcal{E}}\left(\mu\right)=\begin{cases}
\mathcal{E}\left(\sqrt{f},\sqrt{f}\right), & \mu=f\cdot\mathsf{m},\ \sqrt{f}\in\mathcal{F}\\
\infty,
\end{cases}
\]
the Donsker-Varadhan $I$-function as 
\[
I\left(\mu\right)=-\inf_{u\in\mathcal{D}_{+}\left(\mathcal{A}\right),\varepsilon>0}\int_{\mathbb{R}^{d}}\frac{\mathcal{A}u}{u+\varepsilon}d\mu,
\]
the function $I_{\alpha},\alpha>0$ as 
\[
I_{\alpha}\left(\mu\right)=-\inf_{u\in b\mathcal{B}_{+}\left(\mathbb{R}^{d}\right),\varepsilon>0}\int_{\mathbb{R}^{d}}\log\left(\frac{\alpha R_{\alpha}u+\varepsilon}{u+\varepsilon}\right)d\mu.
\]

\begin{lemma} \label{lem:,7.1.2} Let $\mu\in\mathcal{P}$, then
\begin{alignat*}{1}
\underbrace{I_{\alpha}\left(\mu\right)} & \leqq\underbrace{\frac{I\left(\mu\right)}{\alpha}}.\\
-\inf_{u\in b\mathcal{B}_{+}\left(\mathbb{R}^{d}\right),\varepsilon>0}\int_{\mathbb{R}^{d}}\log\left(\frac{\alpha R_{\alpha}u+\varepsilon}{u+\varepsilon}\right)d\mu & \leqq-\frac{1}{\alpha}\inf_{u\in\mathcal{D}_{+}\left(\mathcal{A}\right),\varepsilon>0}\int_{\mathbb{R}^{d}}\frac{\mathcal{A}u}{u+\varepsilon}d\mu
\end{alignat*}
\begin{proof} For $u=R_{\alpha}f\in\mathcal{D}_{+}\left(\mathcal{A}\right)$
and $\varepsilon>0$, let 
\[
\phi\left(\alpha\right)=-\int_{\mathbb{R}^{d}}\log\left(\frac{\alpha R_{\alpha}u+\varepsilon}{u+\varepsilon}\right)d\mu.
\]
From the resolvent equation, we have 
\[
\lim_{\beta\rightarrow\alpha}\frac{R_{\alpha}u-R_{\beta}u}{\alpha-\beta}=-\lim_{\beta\rightarrow\alpha}R_{\alpha}R_{\beta}u=-R_{\alpha}^{2}u,
\]
and 
\[
\frac{d\phi\left(a\right)}{d\alpha}=-\int_{\mathbb{R}^{d}}\frac{R_{\alpha}u-\alpha R_{\alpha}^{2}u}{\alpha R_{\alpha}u+\varepsilon}d\mu=\int_{\mathbb{R}^{d}}\frac{\mathcal{A}R_{\alpha}^{2}u}{\alpha R_{\alpha}u+\varepsilon}d\mu.
\]
Here, if we note that 
\begin{alignat*}{1}
\left(\alpha R_{\alpha}^{2}u-R_{\alpha}u\right)\left(\alpha^{2}R_{\alpha}^{2}u+\varepsilon\right)-\left(\alpha R_{\alpha}^{2}u-R_{\alpha}u\right)\left(\alpha R_{\alpha}^{2}u+\varepsilon\right)\\
=\alpha\left(\alpha R_{\alpha}^{2}u-R_{\alpha}u\right)^{2} & \geqq0
\end{alignat*}
the following inequality, 
\[
\frac{\alpha R_{\alpha}^{2}u-R_{\alpha}u}{\alpha R_{\alpha}u+\varepsilon}\geqq\frac{\alpha R_{\alpha}^{2}u-R_{\alpha}u}{\alpha R_{\alpha}^{2}u+\varepsilon}
\]
holds. Therefore, we derive 
\begin{alignat*}{1}
\int_{\mathbb{R}^{d}}\frac{\mathcal{A}R_{\alpha}^{2}u}{\alpha R_{\alpha}u+\varepsilon}d\mu & \geqq\int_{\mathbb{R}^{d}}\frac{\mathcal{A}R_{\alpha}^{2}u}{\alpha^{2}R_{\alpha}^{2}u+\varepsilon}d\mu=-\frac{1}{\alpha^{2}}\left(-\int_{\mathbb{R}^{d}}\frac{\mathcal{A}R_{\alpha}^{2}u}{R_{\alpha}^{2}u+\frac{\varepsilon}{\alpha^{2}}}d\mu\right)\\
 & \geqq-\frac{1}{\alpha^{2}}\underbrace{\inf_{u\in\mathcal{D}_{+}\left(\mathcal{A}\right),\varepsilon>0}\int_{\mathbb{R}^{d}}\frac{\mathcal{A}u}{u+\varepsilon}d\mu}_{I\left(\mu\right)}.
\end{alignat*}
From this, given $\lim_{\alpha\rightarrow\infty}\alpha R_{\alpha}u\left(x\right)=u\left(x\right)$,
noting that $\lim_{\alpha\rightarrow\infty}\phi\left(\alpha\right)=0$,
we have 
\[
-\phi\left(\alpha\right)=\int_{\alpha}^{\infty}\phi^{\prime}\left(\beta\right)d\beta\geqq-\int_{\alpha}^{\infty}\frac{1}{\beta^{2}}I\left(\mu\right)d\beta=-\frac{1}{\alpha}I\left(\mu\right)
\]
and 
\[
\phi\left(\infty\right)-\phi\left(\alpha\right)=\int_{\mathbb{R}^{d}}\log\left(\frac{\alpha R_{\alpha}u+\varepsilon}{u+\varepsilon}\right)d\mu\geqq\frac{1}{\alpha}\inf_{u\in\mathcal{D}_{+}\left(\mathcal{A}\right),\varepsilon>0}\int_{\mathbb{R}^{d}}\frac{\mathcal{A}u}{u+\varepsilon}d\mu,
\]
to show 
\[
-\inf_{u\in\mathcal{D}_{+}\left(\mathcal{A}\right),\varepsilon>0}\int_{\mathbb{R}^{d}}\log\left(\frac{\alpha R_{\alpha}u+\varepsilon}{u+\varepsilon}\right)d\mu\leqq\underbrace{-\frac{1}{\alpha}\inf_{u\in\mathcal{D}_{+}\left(\mathcal{A}\right),\varepsilon>0}\int_{\mathbb{R}^{d}}\frac{\mathcal{A}u}{u+\varepsilon}d\mu}_{\frac{I\left(\mu\right)}{\alpha}}.
\]

Next, we show 
\[
\inf_{u\in\mathcal{D}_{+}\left(\mathcal{A}\right)}\int_{\mathbb{R}^{d}}\log\left(\frac{\alpha R_{\alpha}u+\varepsilon}{u+\varepsilon}\right)d\mu=\inf_{u\in b\mathcal{B}_{+}\left(\mathbb{R}^{d}\right)}\int_{\mathbb{R}^{d}}\log\left(\frac{\alpha R_{\alpha}u+\varepsilon}{u+\varepsilon}\right)d\mu.
\]
For $g\in bC_{+}\left(\mathbb{R}^{d}\right)$, we have $\left\Vert \beta R_{\beta}f\right\Vert _{\infty}\leqq\left\Vert f\right\Vert _{\infty}$,
\[
0\leqq\beta R_{\beta}f\left(x\right)\rightarrow f\left(x\right),\left(\beta\rightarrow\infty\right)
\]
thus 
\begin{equation}
\int_{\mathbb{R}^{d}}\log\left(\frac{\alpha R_{\alpha}\left(\beta R_{\beta}f\right)+\varepsilon}{\beta R_{\beta}f+\varepsilon}\right)d\mu\overset{\beta\rightarrow\infty}{\rightarrow}\int_{\mathbb{R}^{d}}\log\left(\frac{\alpha R_{\alpha}f+\varepsilon}{f+\varepsilon}\right)d\mu\label{eq:7.1.8}
\end{equation}
holds. We define the measure, $\mu_{\alpha}$, as 
\[
\mu_{\alpha}\left(A\right)=\int_{\mathbb{R}^{d}}\alpha R_{\alpha}\left(x,A\right)d\mu\left(x\right),\ A\in\mathcal{B}\left(\mathbb{R}^{d}\right).
\]
For $v\in b\mathcal{B}_{+}\left(\mathbb{R}^{d}\right)$, we consider
a sequence of functions, $\left\{ g_{n}\right\} _{n=1}^{\infty}\subset bC_{+}\left(\mathbb{R}^{d}\right)\cap L^{2}\left(\mathbb{R}^{d};\mathsf{m}\right)$,
that satisfies, 
\[
\int_{\mathbb{R}^{d}}\left|v-g_{n}\right|d\left(\mu_{\alpha}+\mu\right)\rightarrow0,\ n\rightarrow\infty.
\]
Then, when $n\rightarrow\infty$, we have 
\[
\int_{\mathbb{R}^{d}}\left|\alpha R_{\alpha}v-\alpha R_{\alpha}g_{n}\right|d\mu\leqq\int_{\mathbb{R}^{d}}\alpha R_{\alpha}\left(\left|v-g_{n}\right|\right)d\mu=\int_{\mathbb{R}^{d}}\left|v-g_{n}\right|d\mu_{\alpha}\rightarrow0,
\]
thus the following holds: 
\begin{equation}
\int_{\mathbb{R}^{d}}\log\left(\frac{\alpha R_{\alpha}g_{n}+\varepsilon}{g_{n}+\varepsilon}\right)d\mu\overset{n\rightarrow\infty}{\rightarrow}\int_{\mathbb{R}^{d}}\log\left(\frac{\alpha R_{\alpha}v+\varepsilon}{v+\varepsilon}\right)d\mu.\label{eq:7.1.9}
\end{equation}
From Eq.~(\ref{eq:7.1.8}) and Eq.~(\ref{eq:7.1.9}), we have 
\[
\inf_{u\in\mathcal{D}_{+}\left(\mathcal{A}\right)}\int_{\mathbb{R}^{d}}\log\left(\frac{\alpha R_{\alpha}u+\varepsilon}{u+\varepsilon}\right)d\mu=\inf_{u\in b\mathcal{B}_{+}\left(\mathbb{R}^{d}\right)}\int_{\mathbb{R}^{d}}\log\left(\frac{\alpha R_{\alpha}u+\varepsilon}{u+\varepsilon}\right)d\mu.
\]
\end{proof} \end{lemma}

\begin{lemma} \label{lem:7.1.3} If $I\left(\mu\right)<\infty$,
then $\mu\ll m$. \end{lemma}

If we take a non-negative function that is monotonically increasing,
$f_{n}\in C_{0}\left(\mathbb{R}^{d}\right)$, that point-wise converges
to $f\in bC_{+}\left(\mathbb{R}^{d}\right)$, we have, 
\[
\int_{\mathbb{R}^{d}}\frac{f-\alpha R_{\alpha}f}{R_{\alpha}f+\varepsilon}d\mu=\lim_{n\rightarrow\infty}\int_{\mathbb{R}^{d}}\frac{f_{n}-\alpha R_{\alpha}f_{n}}{R_{\alpha}f_{n}+\varepsilon}d\mu,\ \alpha>0.
\]
Therefore, defining the function, $\mathcal{D}_{+}\left(\hat{\mathcal{A}}\right)$,
as 
\[
\mathcal{D}_{+}\left(\hat{\mathcal{A}}\right)=\left\{ R_{\alpha}f\left|\alpha>0,\ f\in bC_{+}\left(\mathbb{R}^{d}\right),\ f\cancel{\equiv}0\right.\right\} ,
\]
if we set $\hat{\mathcal{A}}\phi=\alpha R_{\alpha}f-f$ with regard
to $\phi=R_{\alpha}f\in\mathcal{D}_{+}\left(\hat{\mathcal{A}}\right)$,
we have the following result. \begin{cor} For $f\in\mathcal{F}$,
\[
\mathcal{E}\left(f,f\right)=\sup_{u\in\mathcal{D}_{+}\left(\hat{\mathcal{A}}\right),\varepsilon>0}\int_{\mathbb{R}^{d}}\frac{-\hat{\mathcal{A}}u}{u+\varepsilon}f^{2}d\mathsf{m}.
\]
\end{cor}

Proof of Theorem~\ref{prop:rate variational}. \begin{proof} We
first show $I\left(\mu\right)\geqq I_{\mathcal{E}}\left(\mu\right)$.
Assume $I\left(\mu\right)<\infty$. Since, $\mu\ll m$, we have, $f=\frac{d\mu}{d\mathsf{m}}$
and let$f^{n}=\sqrt{f}\wedge n$. From $\log\left(1-x\right)\leqq-x$
and $-\infty<\frac{f^{n}-\alpha R_{\alpha}f^{n}}{f^{n}+\varepsilon}<1,\ \left(-\infty<x<1\right)$,
we have 
\[
\int_{\mathbb{R}^{d}}\log\left(\frac{\alpha R_{\alpha}f^{n}+\varepsilon}{f^{n}+\varepsilon}\right)fd\mathsf{m}=\int_{\mathbb{R}^{d}}\log\left(1-\frac{f^{n}-\alpha R_{\alpha}f^{n}}{f^{n}+\varepsilon}\right)fd\mathsf{m}\leqq-\int_{\mathbb{R}^{d}}\frac{f^{n}-\alpha R_{\alpha}f^{n}}{f^{n}+\varepsilon}fd\mathsf{m}.
\]
Thus, 
\begin{alignat}{1}
\int_{\mathbb{R}^{d}}\frac{f^{n}-\alpha R_{\alpha}f^{n}}{f^{n}+\varepsilon}fd\mathsf{m} & \leqq-\int_{\mathbb{R}^{d}}\log\left(\frac{\alpha R_{\alpha}f^{n}+\varepsilon}{f^{n}+\varepsilon}\right)fd\mathsf{m}\label{eq:7.1.11}\\
 & \leqq\underbrace{-\inf_{u\in b\mathcal{B}_{+}\left(\mathbb{R}^{d}\right),\varepsilon>0}\int_{\mathbb{R}^{d}}\log\left(\frac{\alpha R_{\alpha}u+\varepsilon}{u+\varepsilon}\right)d\mu}_{I_{\alpha}\left(f\cdot\mathsf{m}\right)}.\nonumber 
\end{alignat}
For the euqlity, 
\begin{alignat*}{1}
\frac{f^{n}-\alpha R_{\alpha}f^{n}}{f^{n}+\varepsilon}f & =\frac{f^{n}-\alpha R_{\alpha}f^{n}}{f^{n}+\varepsilon}f1_{\left\{ \sqrt{f}\leqq n\right\} }+\frac{f^{n}-\alpha R_{\alpha}f^{n}}{f^{n}+\varepsilon}f1_{\left\{ \sqrt{f}>n\right\} }\\
 & =\underbrace{\frac{\sqrt{f}-\alpha R_{\alpha}f^{n}}{\sqrt{f}+\varepsilon}f1_{\left\{ \sqrt{f}\leqq n\right\} }}_{\left(a\right)}+\underbrace{\frac{n-\alpha R_{\alpha}f^{n}}{n+\varepsilon}f\left\{ \sqrt{f}>n\right\} }_{\left(b\right)},
\end{alignat*}
the absolute value of $\left(a\right)$ and $\left(b\right)$ is evaluated
from above by $\left(\sqrt{f}+\alpha R_{\alpha}f^{n}\right)\sqrt{f}\in L^{1}\left(\mathbb{R}^{d};\mathsf{m}\right)$
and $\frac{n}{n+\varepsilon}f<f\in L^{1}\left(\mathbb{R}^{d};\mathsf{m}\right)$,
respectively. Therefore, from the bounded convergence theorem, we
have 
\[
\lim_{n\rightarrow\infty}\int\frac{f^{n}-\alpha R_{\alpha}f^{n}}{f^{n}+\varepsilon}fd\mathsf{m}=\int_{\mathbb{R}^{d}}\frac{\sqrt{f}-\alpha R_{\alpha}\sqrt{f}}{\sqrt{f}+\varepsilon}fd\mathsf{m}.
\]
When $\varepsilon\rightarrow0$, from Eq.~(\ref{eq:7.1.11}), we
have 
\[
\int_{\mathbb{R}^{d}}\sqrt{f}\left(\sqrt{f}-\alpha R_{\alpha}\sqrt{f}\right)d\mathsf{m}\leqq I_{\alpha}\left(f\cdot\mathsf{m}\right).
\]
Therefore, from Lemma~(\ref{lem:,7.1.2}), we can shown 
\[
\alpha\left\langle \sqrt{f},\sqrt{f}-\alpha R_{\alpha}\sqrt{f}\right\rangle _{\mathsf{m}}\leqq I\left(f\cdot\mathsf{m}\right)<\infty,
\]
and from Lemma~\ref{lem:4.3}, this implies $\sqrt{f}\in\mathcal{F}$
and $\mathcal{E}\left(\sqrt{f},\sqrt{f}\right)\leqq I\left(f\cdot\mathsf{m}\right)$.

Next, we show $I\left(\mu\right)\leqq I_{\mathcal{E}}\left(\mu\right)$.
Let $\phi\in\mathcal{D}_{+}\left(\mathcal{A}\right)$ and define the
semigroup, $P_{t}^{\phi}$, as 
\[
P_{t}^{\phi}f\left(x\right)=\mathbb{E}_{x}\left[\left(\frac{\phi\left(X_{t}\right)+\varepsilon}{\phi\left(X_{0}\right)+\varepsilon}\right)\exp\left(-\int_{0}^{t}\frac{\mathcal{A}\phi}{\phi+\varepsilon}\left(X_{s}\right)ds\right)f\left(X_{t}\right)\right].
\]
Here, $P_{t}^{\phi}$ is $\left(\phi+\varepsilon\right)^{2}m$-symmetric
and satisfies $P_{t}^{\phi}1\leqq1$. For the probability measure,
$\mu=fm\in\mathcal{P}$, such that $\sqrt{f}\in\mathcal{F}$, let
\[
S_{t}^{\phi}\sqrt{f}\left(x\right)=\mathbb{E}_{x}\left[\exp\left(-\int_{0}^{t}\frac{\mathcal{A}\phi}{\phi+\varepsilon}\left(X_{s}\right)ds\right)\sqrt{f}\left(X_{t}\right)\right].
\]
Then, the following holds: 
\begin{alignat*}{1}
\int_{\mathbb{R}^{d}}\left(S_{t}^{\phi}\sqrt{f}\right)^{2}d\mathsf{m} & =\int_{\mathbb{R}^{d}}\left(\phi+\varepsilon\right)^{2}\left(P_{t}^{\phi}\left(\frac{\sqrt{f}}{\phi+\varepsilon}\right)\right)^{2}d\mathsf{m}\\
 & \leqq\int_{\mathbb{R}^{d}}\left(\phi+\varepsilon\right)^{2}P_{t}^{\phi}\left(\left(\frac{\sqrt{f}}{\phi+\varepsilon}\right)^{2}\right)d\mathsf{m}\\
 & \leqq\int_{\mathbb{R}^{d}}\left(\phi+\varepsilon\right)^{2}\left(\frac{\sqrt{f}}{\phi+\varepsilon}\right)^{2}d\mathsf{m}\\
 & =\int_{\mathbb{R}^{d}}fd\mathsf{m}.
\end{alignat*}
Therefore, we have 
\[
0\leqq\lim_{t\rightarrow0}\frac{1}{t}\left\langle \sqrt{f}-S_{t}^{\phi}\sqrt{f},\sqrt{f}\right\rangle _{\mathsf{m}}=\mathcal{E}\left(\sqrt{f},\sqrt{f}\right)+\int_{\mathbb{R}^{d}}\frac{\mathcal{A}\phi}{\phi+\varepsilon}fd\mathsf{m}
\]
and $\mathcal{E}\left(\sqrt{f},\sqrt{f}\right)\geqq I\left(f\cdot\mathsf{m}\right)$
has been shown. \end{proof}

\section{Proof of \ref{Thm:InfiniteDivisible}\label{app:InfiniteDivisible}}

Define the transition semigroup, $T_{t}$, of the Markov process,
$\left\{ X_{t}\right\} $, as a convolution semigroup, 
\[
T_{t}f\left(x\right)=\int_{\mathbb{R}^{d}}f\left(x+y\right)\nu_{t}\left(dy\right),\ t>0,x\in\mathbb{R}^{d},f\in\mathcal{B}_{+}\left(\mathbb{R}^{d}\right).
\]
Here, the characteristic function for $\nu_{t}\left(\cdot\right)$
is given as 
\begin{alignat*}{1}
\mathbb{E}\left[e^{itzY}\right]= & \exp\left[t\left\{ -\frac{1}{2}\left\langle z,Az\right\rangle +\int_{\mathbb{R}^{d}}\left(e^{i\left\langle z,y\right\rangle }-1-\frac{i\left\langle z,y\right\rangle }{1+\left|y\right|^{2}}\right)\nu\left(dy\right)+i\left\langle x,z\right\rangle \right\} \right].
\end{alignat*}
Additionally, $X_{1}\overset{\textrm{d}}{=}X$. Define $\psi\left(z\right)$
as 
\[
\psi\left(z\right)=-\frac{1}{t}\log\mathbb{E}\left[e^{itzY}\right].
\]
For the Fourier transformation of the integrable function, $f$, on
$\mathbb{R}^{d}$, 
\[
\hat{f}\left(z\right)=\frac{1}{\left(2\pi\right)^{d/2}}\int_{\mathbb{R}^{d}}e^{i\left\langle z,y\right\rangle }f\left(y\right)dy,\ z\in\mathbb{R}^{d},
\]
the Parseval formula, 
\[
\left\langle f,g\right\rangle =\left\langle \hat{f},\overline{\hat{g}}\right\rangle ,\ f,g\in L^{2}\left(\mathbb{R}^{d};\mathsf{m}\right),
\]
holds. Here, because $\widehat{T_{t}f}=\hat{\nu}_{t}\cdot\hat{f}$,
we have 
\begin{alignat*}{1}
\mathcal{E}^{\left(t\right)}\left(f,f\right)= & \frac{1}{t}\left\langle f-T_{t}f,f\right\rangle \\
= & \frac{1}{t}\int_{\mathbb{R}^{d}}\left(\hat{f}\left(z\right)-\hat{\nu}_{t}\left(z\right)\hat{f}\left(z\right)\right)\bar{\hat{f}}\left(z\right)dz\\
= & \int_{\mathbb{R}^{d}}\left\Vert \hat{f}\left(z\right)\right\Vert ^{2}\frac{1-\exp\left(-t\psi\left(z\right)\right)}{t}dz.
\end{alignat*}
When $t\downarrow0$, the last integral is monotonically increasing
and converges to $\int_{\mathbb{R}^{d}}\left|\hat{f}\left(z\right)\right|^{2}\psi\left(z\right)dz$.
Therefore, the Dirichlet form associated with the predictive distribution,
$\hat{p}\left(y\left|x\right.\right)$, is 
\begin{align*}
 & \mathcal{F}=f\in L^{2}\left(\mathbb{R}^{d};\mathsf{m}\right):\int_{\mathbb{R}^{d}}\left\Vert \hat{f}\left(z\right)\right\Vert ^{2}\psi\left(z\right)dz<\infty,\\
 & \mathcal{E}\left(f,g\right)=\int_{\mathbb{R}^{d}}\hat{f}\left(z\right)\hat{g}\left(z\right)\psi\left(z\right)dz,\quad f,g\in\mathcal{F}.
\end{align*}

\end{document}